\documentclass[11pt]{amsart}
\usepackage[latin9]{inputenc}
\usepackage{geometry}
\usepackage{units}
\usepackage{mathrsfs}
\usepackage{amsthm}
\usepackage{amstext}
\usepackage{amssymb}
\usepackage{graphicx}
\PassOptionsToPackage{normalem}{ulem}
\usepackage{ulem}

\makeatletter
\numberwithin{equation}{section}
\numberwithin{figure}{section}
\usepackage{enumitem}		
\theoremstyle{plain}
\newtheorem{thm}{\protect\theoremname}[section]
  \theoremstyle{definition}
  \newtheorem{defn}[thm]{\protect\definitionname}
  \theoremstyle{plain}
  \newtheorem{prop}[thm]{\protect\propositionname}
  \theoremstyle{remark}
  \newtheorem{rem}[thm]{\protect\remarkname}
  \theoremstyle{plain}
  \newtheorem{cor}[thm]{\protect\corollaryname}
  \theoremstyle{plain}
  \newtheorem{lem}[thm]{\protect\lemmaname}

\usepackage{amsthm}\usepackage{bm}
\usepackage{mathrsfs}

\makeatother

  \providecommand{\corollaryname}{Corollary}
  \providecommand{\definitionname}{Definition}
  \providecommand{\lemmaname}{Lemma}
  \providecommand{\propositionname}{Proposition}
  \providecommand{\remarkname}{Remark}
\providecommand{\theoremname}{Theorem}

\begin{document}
\global\long\def\la{\lambda}

\global\long\def\Cr{\mathscr{C}\left(f\right)}

\global\long\def\Min{\mathscr{M}_{-}\left(f\right)}

\global\long\def\Max{\mathscr{M}_{+}\left(f\right)}

\global\long\def\Sd{\mathscr{S}\left(f\right)}

\global\long\def\Xt{\mathscr{X}\left(f\right)}

\global\long\def\Nd{N\left(f\right)}

\global\long\def\Z{\mathbb{Z}}

\global\long\def\R{\mathbb{R}}

\global\long\def\De{\Delta}

\global\long\def\na{\nabla}

\title{Topological Properties of Neumann Domains}

\author{{Ram Band and David Fajman}}

\subjclass[2000]{35Pxx, 57M20}
\begin{abstract}
A Laplacian eigenfunction on a two-dimensional manifold dictates some
natural partitions of the manifold; the most apparent one being the
well studied nodal domain partition. An alternative partition is revealed
by considering a set of distinguished gradient flow lines of the eigenfunction
- those which are connected to saddle points. These give rise to Neumann
domains. We establish complementary definitions for Neumann domains
and Neumann lines and use basic Morse homology to prove their fundamental
topological properties. We study the eigenfunction restrictions to
these domains. Their zero set, critical points and spectral properties
allow to discuss some aspects of counting the number of Neumann domains
and estimating their geometry.
\end{abstract}

\keywords{Neumann domains, Neumann lines, nodal domains, Laplacian eigenfunctions,
Morse-Smale complexes}

\maketitle

\section{Introduction}

\noindent Topological properties of Laplacian eigenfunctions on domains
and manifolds are of essential interest to mathematical physics in
recent years \cite{Proc_Kavli,Proc_Wittenberg}. Nodal patterns of
eigenfunctions are a major and well developed research area in this
field. Nodal sets of eigenfunctions have been studied with respect
to their volume \cite{Bruening_mzs78,BruGro_mzs72,ColMin_cmp11,DonFef_im88,HezSog_apde12}
and geometry \cite{BouRud_ahp11,BouRud_inv11} and nodal domains of
eigenfunctions have been studied with respect to their count \cite{BluGnuSmi_prl02,BrFa12,Courant23,GSS05,Klawonn09}
and metric properties \cite{Mangoubi_cpde08,Mangoubi_cmb08,Mangoubi_jta10}.
The study of related notions, called \emph{Neumann lines} and \emph{Neumann
domains} has been recently suggested in two independent works by Zelditch
\cite{Zel_sdg13} and McDonald, Fulling \cite{McDFul_ptrs13}. Neumann
lines and Neumann domains form a partition of the manifold, dictated
by the eigenfunction. The current paper is dedicated to the investigation
of those partitions from topological, geometric and spectral perspectives.
We note that Neumann domains are studied in computational topology
and computer graphics, where they are known as Morse-Smale complexes
and used for applications such as surface segmentation (see \cite{Zomorodian_05}
and references within).

\subsection{\noindent Preliminaries}

\noindent Let $M$ be a 2-dimensional, connected, compact and orientable
surface without boundary with a smooth Riemannian metric $g$ and
let $\Delta_{g}$ be the Laplace-Beltrami operator of $g$. Consider
the eigenvalue problem
\begin{equation}
-\Delta_{g}f=\lambda f.
\end{equation}
We assume in the following that the eigenfunctions $f$ are Morse
functions, i.e.\,have no degenerate critical points. We call such
an $f$ a \emph{Morse-eigenfunction}. In fact, for generic metrics
eigenfunctions are in this class, as shown in \cite{Uhl_ajm76}. The
smooth gradient vector field, $\nabla f$, defines a smooth flow,
$\varphi$, along the integral curves of $-\nabla f$:
\begin{equation}
\begin{aligned} & \varphi:\mathbb{R}\times\,M\rightarrow M,\\
 & \partial_{t}\varphi(t,\,x)=-\nabla f\big|_{\varphi(t,\,x)},\\
 & \varphi(0,\,x)=x.
\end{aligned}
\label{eq:flow}
\end{equation}

\noindent We introduce the following notations. Let $\Cr$ denote
the set of critical points of $f$, $\Sd$ and $\Xt$ the sets of
saddle points and extrema, respectively. In addition, let $\Min$
and $\Max$ denote the sets of minima and maxima of $f$, respectively.

\noindent For a critical point $x\in\Cr$, we denote by $\la_{x}$
its index (the number of negative eigenvalues of the Hessian of $f$
at $x$) and define its stable and unstable manifolds by
\begin{equation}
\begin{aligned}W^{s}(x) & =\{y\in M\big|\lim_{t\rightarrow\infty}\varphi(t,\,y)=x\}\mbox{ and }\\
W^{u}(x) & =\{y\in M\big|\lim_{t\rightarrow-\infty}\varphi(t,\,y)=x\},
\end{aligned}
\end{equation}
respectively. Finally, we recall the following relevant definition.
A \emph{Morse-Smale function} is a Morse function, which in addition
fulfills the Morse-Smale transversality condition, saying that stable
and unstable manifolds intersect transversely (cf.~\cite{BanHur_MorseHomology04}).
In two dimensions the Morse-Smale transversality condition is equivalent
to the condition that there are no two saddle points which are connected
by gradient flow lines. The definition of Neumann domains (definition
\ref{def:ND_using_CW_complex}) already appears in the Morse homology
literature (see e.g., \cite{BanHur_MorseHomology04}). There it is
assumed that the function is Morse-Smale in order to obtain some basic
properties of Neumann domains. However, as there exist eigenfunctions
which are not Morse-Smale, we do not adopt this assumption. Not assuming
this forbids us from relying on existing results which could have
simplified our proofs.

\subsection{Definitions and main results}

In this section we assume $f$ to be a general Morse function. Yet,
some of the results are specialized for Morse eigenfunctions, which
are in the focus of this paper. The following definition is motivated
by Zelditch \cite{Zel_sdg13}.
\begin{defn}
\noindent \label{def:ND_using_CW_complex}A \emph{Neumann domain}
is a connected component of the set
\begin{equation}
\Omega_{p,q}\left(f\right)=W^{s}\left(p\right)\cap W^{u}\left(q\right),
\end{equation}
where $p\in\Min,\,\,q\in\Max$.
\end{defn}
\noindent In the following we omit the indices and denote a Neumann
domain by $\Omega$. The next definition owes to the recent paper
of McDonald and Fulling \cite{McDFul_ptrs13}. We allow a certain
modification of the definition to adapt it to the present approach.
\begin{defn}
\noindent \label{def:ND_using_saddles} The \emph{Neumann line set}
of $f$ is
\begin{equation}
\Nd:=\overline{\bigcup_{{r\in\Sd}}W^{s}(r)\cup W^{u}(r)}.
\end{equation}

\end{defn}
\begin{figure}
\begin{centering}
\hspace{-30mm}\includegraphics[bb=-30bp 0bp 1440bp 785bp,scale=0.18]{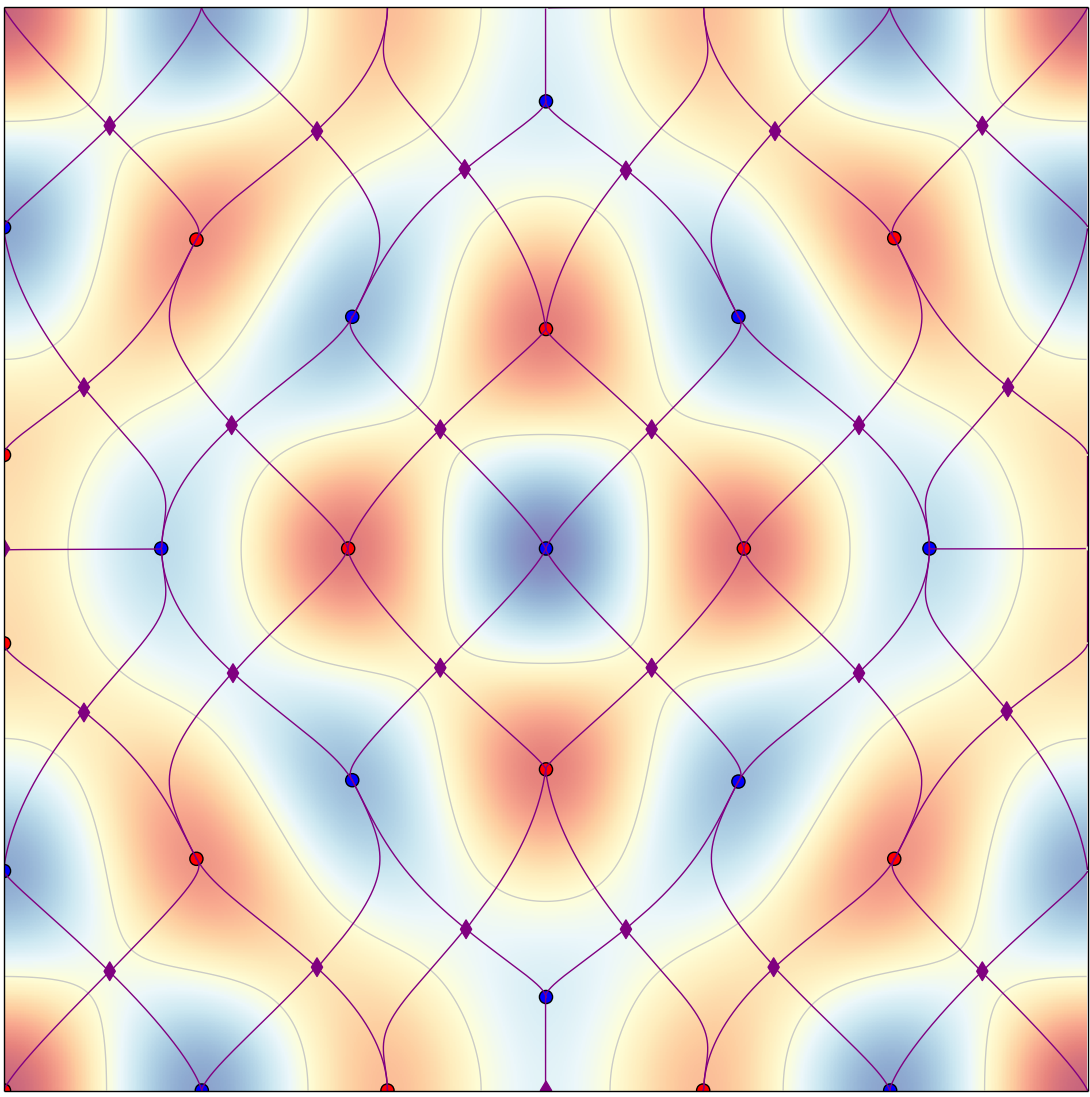}\hspace{-20mm}\includegraphics[bb=-30bp 0bp 1440bp 785bp,scale=0.18]{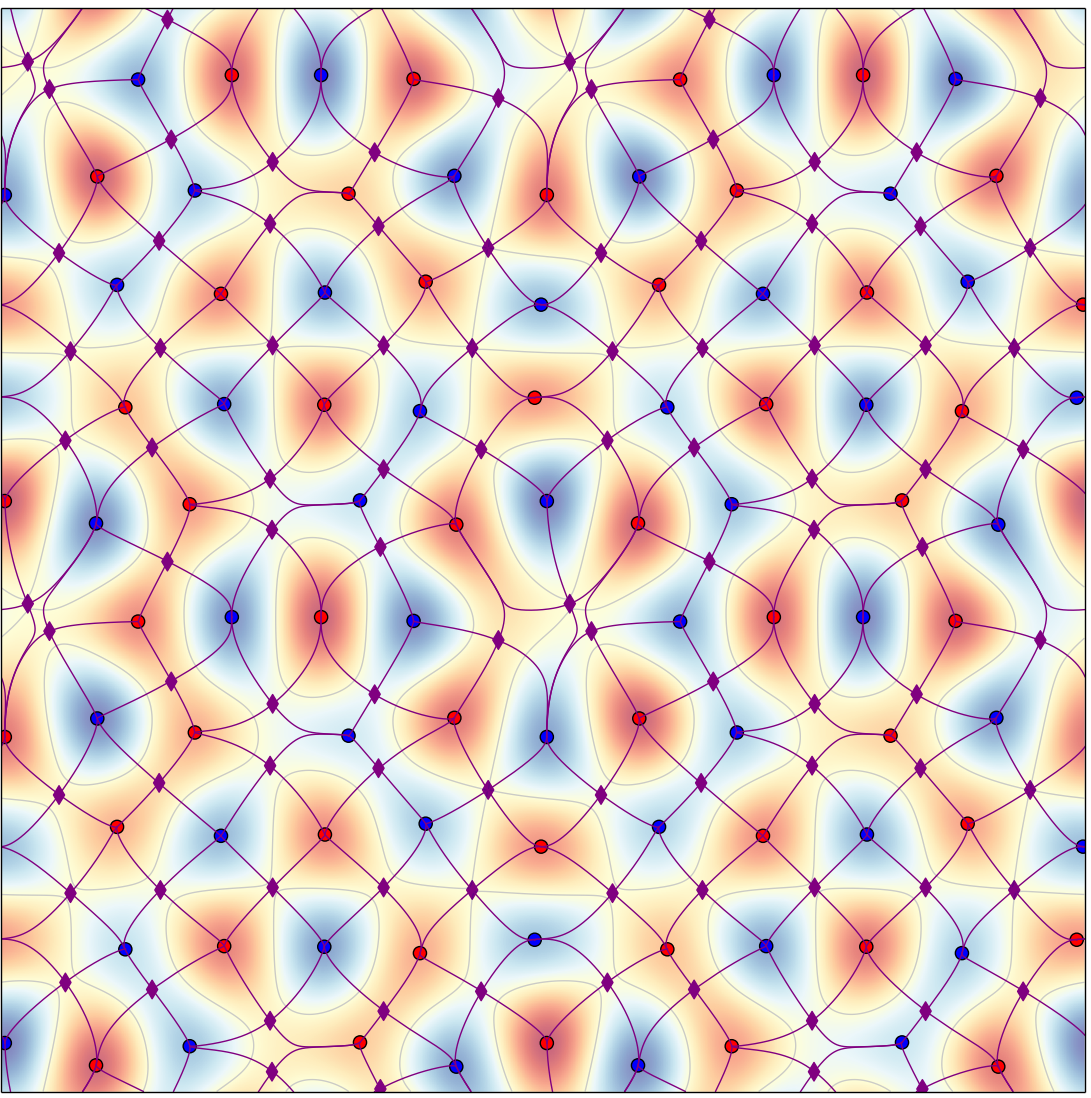}\hspace{-70mm}
\par\end{centering}

\protect\caption{Two eigenfunctions of the eigenvalues $52\pi^{2}$ (left) and $200\pi^{2}$
(right) on the unit flat torus $\mathbb{T}=[0,1]\times[0,1]$. The
nodal domains are colored red and blue and the nodal lines are indicated
by grey lines. Red (blue) circles mark maxima (minima) and purple
diamonds indicate saddle points. The solid lines are the Neumann lines.
\label{fig:eigenfunctions_on_torus}}
\end{figure}


\noindent Figure \ref{fig:eigenfunctions_on_torus} demonstrates the
definitions above by showing the Neumann lines and the Neumann domains
for two eigenfunctions on the flat torus. The next proposition states
that Neumann lines and Neumann domains define a partition of $M$,
assuming that the set of saddle points is not empty.
\begin{prop}
\noindent \label{prop:complementary_property_no_boundary} If $\Nd\neq\emptyset$
then the following disjoint decomposition of the manifold holds.
\begin{equation}
M=\bigsqcup_{{{p\in\Min}\atop {q\in\Max}}}\left\{ \Omega_{p,q}(f)\right\} \,\,\bigsqcup\,\,\Nd
\end{equation}

\end{prop}
The first main result concerns the topological properties of Neumann
domains on closed surfaces.
\begin{thm}
\label{thm:topological_properties_no_boundary} Let $M$ be a smooth,
compact, two dimensional, orientable manifold without boundary and
$g$ a smooth Riemannian metric on $M$. Let $f$ be a Morse function
with $\Sd\neq\emptyset$. Let $p\in\Min,\,q\in\Max$ and $\Omega$
be a connected component of $W^{s}\left(p\right)\cap W^{u}\left(q\right)$,
i.e., $\Omega$ is a Neumann domain. The following properties hold.\global\long\def\theenumi{\roman{enumi}}

\subparagraph*{\uline{Critical points location}\global\long\def\theenumi{\roman{enumi}}
}
\begin{enumerate}[resume]
\item $\Cr\subset\Nd$\label{enu:thm-no-bdry-critical-points}
\item $\Xt\cap\partial\Omega=\{p,q\}$\label{enu:thm-no-bdry-extremal-points}
\item If $f$ is in addition a Morse-Smale function then $\partial\Omega$
consists of Neumann lines connecting saddle points with extrema. In
particular, the boundary, $\partial\Omega$, contains either one or
two saddle points. \label{enu:thm-no-bdry-Morse-Smale}
\end{enumerate}

\subparagraph*{\uline{Neumann domain topology}\global\long\def\theenumi{\roman{enumi}}
}
\begin{enumerate}[resume]
\item $\Omega$ is a simply connected open set.\label{enu:thm-no-bdry-simply-connected}
\end{enumerate}

\subparagraph*{\uline{Level sets of $\left.f\right|_{\Omega}$}\global\long\def\theenumi{\roman{enumi}}
}

Let $c\in\left(f\left(p\right),f\left(q\right)\right)\subset\R$.
\begin{enumerate}[resume]
\item $\Omega\cap f^{-1}\left(c\right)\neq\emptyset$\label{enu:thm-no-bdry-level-sets-1}
\item Each connected component of $\overline{\Omega}\cap f^{-1}\left(c\right)$
has a non empty intersection with $\partial\Omega$.\label{enu:thm-no-bdry-level-sets-2}
\item $\overline{\Omega}\cap f^{-1}\left(c\right)$ is an embedding of a
closed one dimensional interval, without self-intersections, and it
intersects $\partial\Omega$ only at its two endpoints. .\label{enu:thm-no-bdry-level-sets-3}
\end{enumerate}
\end{thm}
We conclude that all Neumann domains are simply connected, which is
a fundamental difference to nodal domains. Moreover, all critical
points are located on the Neumann lines and the boundary of each Neumann
domain contains precisely one minimum and one maximum. Although the
theorem is stated for general Morse functions, in the sequel we apply
the theorem to Morse-eigenfunctions. Under this further assumption,
the nodal set plays an important role as we clarify below.
\begin{rem}
\label{rem:nodal_set_of_Neumann_domain} The maxima of an eigenfunction
are positive and its minima are negative. Therefore, applying theorem
\ref{thm:topological_properties_no_boundary} to a Morse-eigenfunction,
we may choose the value $c=0$ in \eqref{enu:thm-no-bdry-level-sets-1},
\eqref{enu:thm-no-bdry-level-sets-2}, \eqref{enu:thm-no-bdry-level-sets-3}
of the theorem for all Neumann domains. This yields that the intersection
of a Neumann domain with the nodal set is a non-self-intersecting
curve touching the Neumann domain boundary at two endpoints.
\begin{figure}[htbp]
\begin{centering}
\includegraphics[scale=0.28]{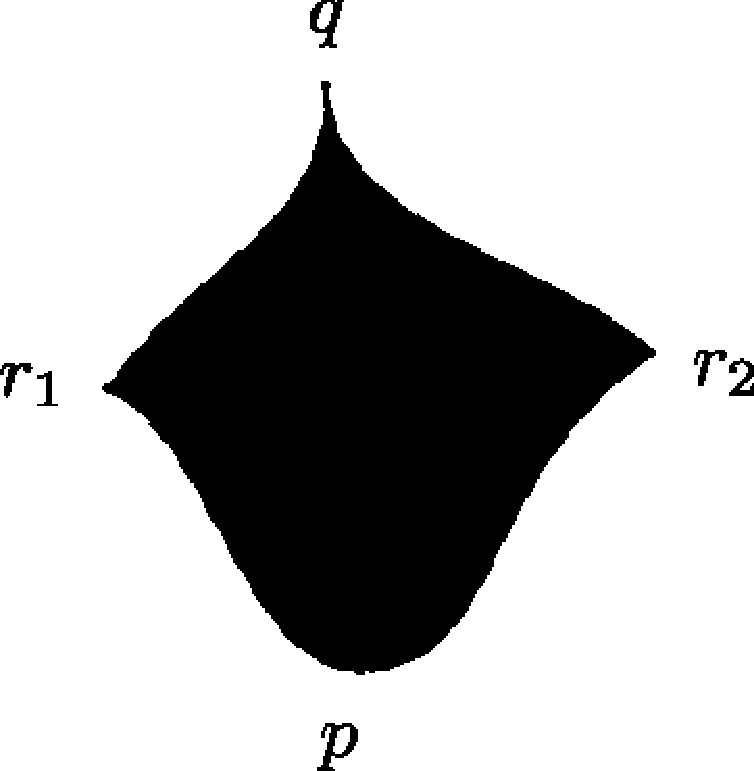}
\par\end{centering}

\protect\caption{The topological structure of a Neumann domain of a Morse-Smale eigenfunction.
The dashed lines mark the gradient flow lines forming the boundary
of the Neumann domain and the solid line marks the nodal line of the
eigenfunction restricted to the Neumann domain. The eigenfunction
critical points on the boundary are indicated by $q$, $p$, $r_{1}$,
$r_{2}$ (maximum, minimum and saddle points, respectively).}

\label{fig-gen str ND}
\end{figure}

\end{rem}
The generic structure of Neumann domains for a Morse-Smale eigenfunctions
which results from theorem \ref{thm:topological_properties_no_boundary}
is displayed in figure \ref{fig-gen str ND}. The theorem, in particular,
allows to bound the number of Neumann domains in terms of the nodal
domain count.
\begin{cor}
\label{cor:Neumann_count_bdd_by_nodal_count}

Let $(M,g)$ be as in theorem \ref{thm:topological_properties_no_boundary}
and $f$ a Morse-eigenfunction on $M$. Let $\mu$ denote the number
of Neumann domains of $f$ and $\nu$ denote the number of its nodal
domains. Then $2\mu\geq\nu$.

\end{cor}
\noindent We proceed by discussing a fundamental spectral property
of Neumann domains. By proposition \ref{prop:complementary_property_no_boundary}
the boundary of a Neumann domain $\Omega$ of an eigenfunction $f$
consists of Neumann lines, which are particular gradient flow lines,
$\left\{ \varphi\left(t,x\right)\right\} _{t\in\R}$ (see for instance
in proof of lemma \ref{lem:boundary_of_stble_mnfld_of_saddle}). Hence
the normal derivative of $f$ at $\partial\Omega$ vanishes. We conclude
that $\left.f\right|_{\Omega}$ is an eigenfunction on $\Omega$ with
Neumann boundary conditions. Hence the name Neumann domains, which
was coined in \cite{McDFul_ptrs13}. A natural question concerns the
position of $\left.f\right|_{\Omega}$ in the spectrum of $\Omega$.
For a nodal domain, $D$, the answer to this question is trivial,
as $\left.f\right|_{D}$ corresponds to the first eigenvalue in the
Dirichlet spectrum of $D$. This observation is a key ingredient in
various nodal domain count results, a fundamental of which is Pleijel's
\cite{Pleijel56}. Similar results for the Neumann domain count may
be obtained from estimating the position of $\left.f\right|_{\Omega}$
in the Neumann spectrum of $\Omega$. Theorem \ref{thm:topological_properties_no_boundary}
\eqref{enu:thm-no-bdry-critical-points}, \eqref{enu:thm-no-bdry-extremal-points},
\eqref{enu:thm-no-bdry-simply-connected}, \eqref{enu:thm-no-bdry-level-sets-3}
suggest that $\left.f\right|_{\Omega}$ cannot have too rich a structure.
This may lead to conjecture that there exists a positive $n\in\mathbb{N}$
such that for every Neumann domain $\Omega$, the restricted eigenfunction,
$\left.f\right|_{\Omega}$, is at most the $n$-th eigenfunction (of
the restricted eigenproblem). Indeed, in \cite{Zel_sdg13} it is suggested
that ``possibly it is 'often' the first non-constant Neumann eigenfunction''.
The following proposition constitutes a counter-example. For a domain
$\Omega$ and an eigenfunction $f$ on $\Omega$ satisfying Neumann
boundary conditions, we denote by $\textrm{pos}\left(f,\Omega\right)$
the position of $f$ in the spectrum of $\Omega$. We set the position
of the trivial constant function to be $\textrm{pos}\left(\textrm{const},\Omega\right)=0$
and for degenerate eigenvalues the position is chosen to be the minimal
one.
\begin{prop}
\label{prop:unbounded_spectral_position} Let $\mathbb{T}$ be the
standard flat two-dimensional torus. There exists a sequence $\left\{ f_{k}\right\} _{k=1}^{\infty}$
of Laplacian eigenfunctions on $\mathbb{T}$ with $\left\{ \Omega_{k}\right\} _{k=1}^{\infty}$
a sequence of Neumann domains, $\Omega_{k}$ being a Neumann domain
of $f_{k}$, such that the sequence $\left\{ \textrm{pos}(\left.f_{k}\right|_{\Omega_{k}},\Omega_{k})\right\} _{k=1}^{\infty}$
is unbounded.
\end{prop}
\noindent The paper is structured as following. In the next section
we treat manifolds without boundary and prove proposition \ref{prop:complementary_property_no_boundary}
and theorem \ref{thm:topological_properties_no_boundary}. In section
\ref{sec:with_boundary} we prove analogous results for manifolds
with Dirichlet boundary (proposition \ref{prop:complementary_property_with_boundary}
and theorem \ref{thm-2}). In section \ref{sec:geom_spec_properties}
we present some geometric and spectral properties of Neumann domains;
we estimate the diameter of Neumann domains (theorem \ref{thm-lower_bound_on_outer_radius}),
prove proposition \ref{prop:unbounded_spectral_position} and discuss
the counter-example which stands behind it. Finally, in section \ref{sec:number_of_Neumann_domains}
we prove corollary \ref{cor:Neumann_count_bdd_by_nodal_count} and
its analogue for the boundary case (corollary \ref{cor:Neumann_count_bdd_by_nodal_count-for_boundary}),
and relate the number of Neumann domains to the number of critical
points and their degrees.

\section{The structure of Neumann lines and Neumann domains - manifolds without
boundary\label{sec:no_boundary}}

\subsection{Some basics in Morse theory}

\noindent A fundamental theorem in Morse homology is the stable/unstable
manifold theorem, part of which we quote here.
\begin{lem}
{[}part of theorem 4.2 in \cite{BanHur_MorseHomology04}{]} \label{lem:stable_mnfold_are_open_sets}
Let $f$ be a real Morse function on an $m-$dimensional, compact
Riemannian manifold and let $x$ be a critical point of $f$. The
stable and unstable manifolds of $x$ are smoothly embedded open disks
of dimension $m-\la_{x}$ and $\la_{x}$, respectively, where $\lambda_{x}$
is the Morse index of $f$ at $x$.
\end{lem}
\noindent In two dimensions the non-degenerate critical points are
maxima, minima and saddle points. According to the lemma above the
stable manifold of a maximum $q\in M$ is $\{q\}$, and the unstable
manifold is the embedding of a two dimensional open disk. The converse
holds for a minimum. For a saddle point, $r\in M$, the stable and
unstable manifolds are embeddings of open one dimensional intervals.
Another useful tool is the decomposition of the manifold $M$ into
a union of stable (or unstable) manifolds.
\begin{lem}
{[}Proposition 4.22 in \cite{BanHur_MorseHomology04}{]}\label{lem:decomp_to_stble_mnfolds}
Let $f:M\rightarrow\mathbb{R}$ be a Morse function on a compact,
smooth, closed Riemannian manifold $(M,g)$, then $M$ is a disjoint
union of the stable manifolds of $f$, i.e.\,
\begin{equation}
M=\bigsqcup_{x\in\Cr}W^{s}(x).
\end{equation}
Similarly,
\begin{equation}
M=\bigsqcup_{x\in\Cr}W^{u}(x).
\end{equation}

\end{lem}

\subsection{Proofs of lemmata}

Throughout this section we assume that $M$ is a 2-dimensional compact,
orientable surface without boundary and $f$ is a Morse function on
$M$. As saddle points play a major role in defining the Neumann line
set, $\Nd$ (see definition \ref{def:ND_using_saddles}), it is useful
in what follows to understand the local behavior of $\Nd$ in the
vicinity of saddle points. The following lemma summarizes results
of that kind, some of which appear in \cite{McDFul_ptrs13}.
\begin{lem}
\label{lem:local_Neumann_set_near_saddle}

Let $M$ and $f$ be as above and let $r\in\Sd$. Then
\begin{enumerate}
\item \label{enu:lem_local_Neumann_near_saddle-1}There exists a neighborhood
$U$ of $r$ such that $\Nd\cap U$ consists of four curves which
meet with right angles at $r$.
\item \label{enu:lem_local_Neumann_near_saddle-2}There exists a neighborhood
$V$ of $r$ such that the previous claim holds in $V$ and in addition,
$f^{-1}\left(f\left(r\right)\right)\cap V$ consists of four curves
which meet at $r$ and interlace with the four curves $\Nd\cap U$.
\end{enumerate}
\end{lem}
\begin{rem}
The case $f\left(r\right)=0$ is particularly interesting as it relates
the nodal lines and the Neumann lines in the vicinity of $r$.\end{rem}
\begin{proof}
The first claim of the lemma is proved in \cite{McDFul_ptrs13} by
examining the Taylor expansion of $f$ around $r$. The second claim
follows similarly.
\end{proof}

\noindent We start by providing three basic lemmata which are required
for proving that Neumann lines and Neumann domains form complementary
sets (cf.~proposition \ref{prop:complementary_property_no_boundary}).
\begin{lem}
{[}Proposition 3.19 in \cite{BanHur_MorseHomology04}{]}\label{lem:flow_lim_is_critic_pt}
$\forall x\in M$, both limits $\lim_{t\rightarrow\pm\infty}\varphi\left(t,x\right)$
exist and they are both critical points of $f$, i.e., $\lim_{t\rightarrow\pm\infty}\varphi\left(t,x\right)\in\Cr$.
\end{lem}
~
\begin{lem}
\label{lem:boundary_of_stble_mnfld_of_saddle} Let $r\in\Sd.$ Then
$q\in\overline{W^{s}\left(r\right)}\backslash W^{s}\left(r\right)$
if and only if $q\in\Cr$ and $W^{s}\left(r\right)\cap W^{u}\left(q\right)\neq\emptyset$.
Similarly, $p\in\overline{W^{u}\left(r\right)}\backslash W^{u}\left(r\right)$
if and only if $p\in\Cr$ and $W^{u}\left(r\right)\cap W^{s}\left(p\right)\neq\emptyset$.\end{lem}
\begin{proof}
We start by proving the direction $\left(\Rightarrow\right)$. By
lemma \ref{lem:stable_mnfold_are_open_sets} we know that $W^{s}\left(r\right)$
is homeomorphic to an embedded open one dimensional interval. Let
$x_{1},x_{2}\in W^{s}\left(r\right)$ two points in different connected
components of $W^{s}(r)\setminus\{r\}$. Each of the sets $X_{1}:=\left\{ \varphi\left(t,x_{1}\right)\right\} _{t\in\R},\,X_{2}:=\left\{ \varphi\left(t,x_{2}\right)\right\} _{t\in\R}$
is also homeomorphic to an embedded open one dimensional interval,
and we have the disjoint decomposition $W^{s}\left(r\right)=X_{1}\cup\left\{ r\right\} \cup X_{2}$.
As $\lim_{t\rightarrow\infty}\varphi\left(t,x_{1,2}\right)=r$ we
get that $\overline{W^{s}\left(r\right)}\backslash W^{s}\left(r\right)=\left\{ \lim_{t\rightarrow-\infty}\varphi\left(t,x_{1}\right),\,\lim_{t\rightarrow-\infty}\varphi\left(t,x_{2}\right)\right\} $.
In particular we conclude that
\begin{equation}
q\in\overline{W^{s}\left(r\right)}\backslash W^{s}\left(r\right)\,\Leftrightarrow\,q\in\left\{ \lim_{t\rightarrow-\infty}\varphi\left(t,x_{1}\right)\,,\,\lim_{t\rightarrow-\infty}\varphi\left(t,x_{2}\right)\right\} .
\end{equation}
Combining this with the implication
\begin{equation}
q\in\left\{ \lim_{t\rightarrow-\infty}\varphi\left(t,x_{1}\right)\,,\,\lim_{t\rightarrow-\infty}\varphi\left(t,x_{2}\right)\right\} \Rightarrow\,q\in\Cr\,\land\,\left\{ x_{1}\in W^{u}\left(q\right)\,\lor\,x_{2}\in W^{u}\left(q\right)\right\} ,
\end{equation}
which follows from lemma \ref{lem:flow_lim_is_critic_pt} we get
\begin{equation}
q\in\overline{W^{s}\left(r\right)}\backslash W^{s}\left(r\right)\,\Rightarrow\,q\in\Cr\,\land\,W^{s}\left(r\right)\cap W^{u}\left(q\right)\neq\emptyset.
\end{equation}
For the other direction, we choose some $x\in W^{s}\left(r\right)\cap W^{u}\left(q\right)$.
For any sequence, $\left\{ t_{n}\right\} $ such that $t_{n}\rightarrow-\infty$
we get that $\varphi\left(t_{n},x\right)\rightarrow q$ and thus $q\in\overline{W^{s}\left(r\right)}$
as an accumulation point of a sequence in $W^{s}\left(r\right)$.
The assumption $q\in\Cr$ gives $\left.\na f\right|_{q}=0$, which
implies $q\notin W^{s}\left(r\right)$, so that $q\in\overline{W^{s}\left(r\right)}\backslash W^{s}\left(r\right)$.
The second part of the lemma is proven similarly. \end{proof}
\begin{lem}
\label{lem:Neumann_nodes_contain_all_extrema}If $\Nd\neq\emptyset$
then $\Nd=\left\{ \bigcup_{{r\in\Sd}}W^{s}(r)\cup W^{u}(r)\right\} \bigsqcup\Xt$.\end{lem}
\begin{proof}
We observe that
\begin{equation}
\begin{aligned}\Nd & \subseteq\left\{ \bigcup_{{r\in\Sd}}W^{s}(r)\cup W^{u}(r)\right\} \bigcup\Cr\\
 & =\left\{ \bigcup_{{r\in\Sd}}W^{s}(r)\cup W^{u}(r)\right\} \bigsqcup\Xt,
\end{aligned}
\end{equation}
where the first line is a deduction from lemma \ref{lem:boundary_of_stble_mnfld_of_saddle}
and the second holds as $\Cr\backslash\Xt=\Sd\subset\bigcup_{{r\in\Sd}}W^{s}(r)$.

We proceed to show that the inclusion above is an exact equality.
Let $q\in M$ be a maximum of $f$. We show that if $\Nd\neq\emptyset$
then $\exists r\in\Sd$ such that $q\in\overline{W^{s}(r)}$. Similar
arguments can be used to show that if $p\in M$ is a minimum of $f$
then $\exists r\in\Sd$ such that $p\in\overline{W^{u}(r)}$ and in
combination this proves the lemma. We consider now the maximum $q\in M$
in view of the second decomposition stated in lemma \ref{lem:decomp_to_stble_mnfolds}.
According to it, $M$ can be decomposed into \emph{(a)} stable manifolds
of minima, which are two dimensional simply connected subsets of $M$,
\emph{(b)} stable manifolds of the saddle points, which are open one-dimensional
subsets of $M$ and (\emph{c}) the set of all maxima. We assume that
there is no saddle point, such that $q$ belongs to the closure of
its stable manifold and show that this implies $\Nd=\emptyset$. By
the assumption and lemma \ref{lem:boundary_of_stble_mnfld_of_saddle},
there is an open neighborhood $U$ of $q$, which does not intersect
with stable manifolds of saddle points and does not contain any other
maxima. By the decomposition of $M$ from lemma \ref{lem:decomp_to_stble_mnfolds},
the punctured neighborhood, $U\setminus\left\{ q\right\} $, can be
covered by a finite number of stable manifolds of minima. However,
as these stable manifolds are open and disjoint this is only possible
if $U\setminus\left\{ q\right\} $ is covered by exactly one stable
manifold of some minimum $p$. As $W^{s}\left(p\right)$ is homeomorphic
to an open two-dimensional disk we conclude that $q$ is a single
connected component of this stable manifold's boundary, $\partial W^{s}\left(p\right)$
and this implies that $\overline{W^{s}\left(p\right)}=M$ and $M=S^{2}$.
In particular, this leaves no saddle points of $f$ on $M$ and therefore
$\Nd=\emptyset$.
\end{proof}

\subsection{Proofs of proposition \ref{prop:complementary_property_no_boundary}
and theorem \ref{thm:topological_properties_no_boundary}}

\noindent Following lemma \ref{lem:Neumann_nodes_contain_all_extrema},
as long as the set of Neumann lines is non-empty, we get that it is
complementary to the union of the Neumann domains, which is the statement
of proposition \ref{prop:complementary_property_no_boundary}, proven
below.
\begin{proof}
\noindent [of proposition \ref{prop:complementary_property_no_boundary}]
Note the following disjoint decomposition of the manifold
\begin{equation}
M=\Big\{\bigsqcup_{{{p\in\Min}\atop {q\in\Max}}}\Big[W^{s}\left(p\right)\cap W^{u}\left(q\right)\Big]\Big\}\bigsqcup\Big\{\bigsqcup_{{{r\in\Sd}\atop {}}}\Big\{ W^{s}(r)\cup W^{u}(r)\Big\}\Big\}\bigsqcup\Xt.
\end{equation}
One can check the validity of this decomposition by separation to
cases. For every $x\in M$, we get from lemma \ref{lem:flow_lim_is_critic_pt}
that $\lim_{t\rightarrow\pm\infty}\varphi\left(t,x\right)\in\Cr$.
If $x$ is a critical point itself then both limits are equal to $x$
and $x\in\Cr=\Xt\backslash\Sd$. Otherwise, if both limits ($\lim_{t\rightarrow\pm\infty}\varphi\left(t,x\right)$)
are different and they are obtained at extremal points then
\begin{equation}
x\in\bigsqcup_{\begin{array}{c}
{{p\in\Min}\atop {q\in\Max}}\end{array}}\left[W^{s}\left(p\right)\cap W^{u}\left(q\right)\right].
\end{equation}
Finally, there is also the case where at least one of the limits is
obtained at a saddle point and then $x\in\bigsqcup_{{r\in\Sd}}\left\{ W^{s}(r)\cup W^{u}(r)\right\} $.
The proposition is now proven as the last two terms of the union equal
$\Nd$ by lemma \ref{lem:Neumann_nodes_contain_all_extrema}.\end{proof}
\begin{rem}
Let $M$ be a two-dimensional manifold as above with genus $g$. Let
$f$ be a Morse function on $M$ with no Neumann lines, $\Nd=\emptyset$.
From the equivalence $\Nd=\emptyset\Leftrightarrow\Sd=\emptyset$
and from Morse inequalities, $2-2g=\left|\Max\right|-\left|\Sd\right|+\left|\Min\right|$
we deduce that $g=0$ and $f$ has a single maximum and a single minimum.
In this case $f$ has a single Neumann domain and the only points
in $M$ not belonging to it are the two extrema. All other cases ($\Nd\neq\emptyset$)
are treated by proposition \ref{prop:complementary_property_no_boundary}.\end{rem}
\begin{proof}
\uline{[of theorem \mbox{\ref{thm:topological_properties_no_boundary}}]}~\\
{[}\eqref{enu:thm-no-bdry-critical-points}{]}. From lemma \ref{lem:Neumann_nodes_contain_all_extrema},
we deduce $\Xt\subset\Nd$. In addition, $\Sd\subset\Nd$ by definition.\\
 {[}\eqref{enu:thm-no-bdry-extremal-points}{]} First we show that
$p,q\in\partial\Omega$. Since $p,q\notin\Omega$ it suffices to show
that $p,q\in\overline{\Omega}$. Start from any $x\in\Omega$. Consider
the flow line which passes through $x$, $X=\left\{ \varphi(t,x)\right\} _{t\in\mathbb{R}}$.
As $x\in W^{s}\left(p\right)$, we get by definition that $X\subset W^{s}\left(p\right)$.
Similarly, $X\subset W^{u}\left(q\right)$ and therefore $X\subset\Omega$.
As $\lim_{t\rightarrow\infty}\varphi\left(t,x\right)=p,$ each neighborhood
of $p$ has a non-empty intersection with $X$ (and hence with $\Omega$)
and therefore $p\in\partial\Omega$. A similar argument shows that
$q\in\partial\Omega$. Now assume by contradiction that there is some
other minimum, $\tilde{p}\neq p$ such that $\tilde{p}\in\partial\Omega$.
Being on the boundary, we have that $W^{s}\left(\tilde{p}\right)\cap\Omega\neq\emptyset$.
From the definition of $\Omega$, we get $W^{s}\left(\tilde{p}\right)\cap W^{s}\left(p\right)\neq\emptyset$,
which gives a contradiction. A similar argument shows that $q$ is
the only maximum of $f$ which belongs to $\partial\Omega$.\\
 {[}\eqref{enu:thm-no-bdry-Morse-Smale}{]} This is an immediate deduction
from the definition of a Morse-Smale function.\\
 {[}\eqref{enu:thm-no-bdry-simply-connected}{]} $\Omega$ is open
being the intersection of two open sets (lemma \ref{lem:stable_mnfold_are_open_sets})
or a connected component of such intersection. Examine the following
sequence of homomorphisms between homology groups $H_{n}$, $H_{n-1}$.
\begin{equation}
H_{n}\left(W^{s}\left(p\right)\cup W^{u}\left(q\right)\right)\longrightarrow H_{n-1}\left(W^{s}\left(p\right)\cap W^{u}\left(q\right)\right)\longrightarrow H_{n-1}\left(W^{s}\left(p\right)\right)\oplus H_{n-1}\left(W^{u}\left(q\right)\right).
\end{equation}
This sequence is exact, being part of the Mayer-Vietoris sequence
(cf.~\cite{BoTu82}) and using that the sets $W^{s}\left(p\right),\,W^{u}\left(q\right)$
are open. For $n\geq2$ we have that $H_{n}\left(W^{s}\left(p\right)\cup W^{u}\left(q\right)\right)=0$
as $M$ is two-dimensional and also as $W^{s}\left(p\right)\cup W^{u}\left(q\right)\subsetneq M$
(which holds as $\Sd\neq\emptyset$). For $n\geq2$ we also have that
$H_{n-1}\left(W^{s}\left(p\right)\right)=H_{n-1}\left(W^{u}\left(q\right)\right)=0$,
as $W^{s}\left(p\right),\,W^{u}\left(q\right)$ are both embeddings
of a two dimensional open disk, by lemma \ref{lem:stable_mnfold_are_open_sets}.
We thus conclude from the exact sequence above that $H_{n-1}\left(W^{s}\left(p\right)\cap W^{u}\left(q\right)\right)=0$
for $n\geq2$. In particular, for $n=2$ we conclude that, $\Omega$
is simply connected if it is path connected. A Neumann domain is indeed
path connected, as $W^{s}\left(p\right)$ and $W^{u}\left(q\right)$
are smooth embeddings of two-dimensional disks, by lemma \ref{lem:stable_mnfold_are_open_sets}.\\
 {[}\eqref{enu:thm-no-bdry-level-sets-1}{]} As $c\in\left(f\left(p\right),f\left(q\right)\right)$,
by continuity, $f$ must obtain the value $c$ somewhere in $\Omega$.\\
 {[}\eqref{enu:thm-no-bdry-level-sets-2}{]} Assume by contradiction
that there is a connected component of $\overline{\Omega}\cap f^{-1}\left(c\right)$
which does not intersect $\partial\Omega$. As $\Omega$ is simply
connected, this means that there is a subdomain $\omega\subset\Omega$
such that either $\left.f\right|_{\omega}\geq c$ or $\left.f\right|_{\omega}\leq c$
and $\left.f\right|_{\partial\omega}=c$. $f$ cannot be identically
equal to $c$ in $\omega$, being a Morse function, and therefore
there is either a maximum or minimum of $f$ inside $\omega$, which
contradicts \eqref{enu:thm-no-bdry-critical-points}.\\
 {[}\eqref{enu:thm-no-bdry-level-sets-3}{]} We deduce from \eqref{enu:thm-no-bdry-simply-connected}
that $\partial\Omega$ has a single connected component. This boundary,
$\partial\Omega$, decomposes to two curves, $\gamma_{1},\gamma_{2}$,
whose endpoints are $p,q$. Namely, $\partial\Omega=\gamma_{1}\cup\gamma_{2}$
and $\gamma_{1}\cap\gamma_{2}=\left\{ p,q\right\} $. The restriction,
$\left.f\right|_{\partial\Omega}$ is monotonic on $\gamma_{1}$ and
$\gamma_{2}$. As $c\in\left(f\left(p\right),f\left(q\right)\right)$
we conclude that $\left.f\right|_{\partial\Omega}$ equals $c$ at
exactly two points, $x\in\gamma_{1},\,y\in\gamma_{2}$. By \eqref{enu:thm-no-bdry-critical-points}
$f^{-1}\left(c\right)$ has no critical points in $\Omega$, so we
deduce from the inverse function theorem that $f^{-1}\left(c\right)$
is union of one-dimensional non intersecting curves. The endpoints
of these curves are $x,y$. Yet, if there is more than one curve in
this union, this implies the existence of a subdomain $\omega\subset\Omega$
such that either $\left.f\right|_{\omega}\geq c$ or $\left.f\right|_{\omega}\leq c$
and $\left.f\right|_{\partial\omega}=c$. This was already ruled out
in \eqref{enu:thm-no-bdry-level-sets-2}.
\end{proof}
\noindent Theorem \ref{thm:topological_properties_no_boundary} implies
that the eigenfunction restriction to a Neumann domain, $\left.f\right|_{\Omega}$,
has a relatively simple structure. According to claim \eqref{enu:thm-no-bdry-critical-points},
$\left.f\right|_{\Omega}$ does not have any critical points. Claim
\eqref{enu:thm-no-bdry-extremal-points} shows that there are only
two extremal points of $\left.f\right|_{\Omega}$, which lie on the
boundary, $\partial\Omega$, and they are exactly the defining minimum
and maximum, $p,q$, of the Neumann domain, $\Omega_{p,q}=W^{s}\left(p\right)\cap W^{u}\left(q\right)$.
According to claim \eqref{enu:thm-no-bdry-simply-connected} a Neumann
domain, $\Omega$, is simply connected and this is used in proving
claims \eqref{enu:thm-no-bdry-level-sets-1}-\eqref{enu:thm-no-bdry-level-sets-3},
which deal with the level set contained within $\Omega$. By claim
\eqref{enu:thm-no-bdry-level-sets-3}, the level pattern of $\left.f\right|_{\overline{\Omega}}$
is simple; it is a single line without self intersections and with
two endpoints on the boundary, $\partial\overline{\Omega}$. For the
additional assumption of Morse-Smale, a typical structure of $\left.f\right|_{\overline{\Omega}}$
is demonstrated in figure \ref{fig-gen str ND}. For a Morse function
which is not Morse-Smale it is possible that there are more than two
saddle points on the boundary of the Neumann domain.\\

\section{Manifolds with Dirichlet boundary\label{sec:with_boundary}}

\noindent In this section we discuss the structure of Neumann domains
on surfaces with boundary. The manifolds, $M$, which we consider
are simply-connected, compact subsets of a compact, closed 2-dimensional
smooth manifold $\mathbf{M}$. We assume that $M$ has a piecewise
smooth boundary with Dirichlet boundary conditions. If $M$ has angles
we assume that they are all non-zero. Many of the explicit examples
which are used to study the characteristic structures of eigenfunctions
are of this type (for example billiards \cite{BluGnuSmi_prl02}).
A particular complication that arises for the case with boundary is
due to the fact that the structure of stable and unstable manifolds
at the boundary is not easily accessible in general. To circumvent
this issue we restrict our study to a class of eigenfunctions which
we introduce in the following.
\begin{defn}
Let $f:\overline{M}\rightarrow\R$ be a Morse-eigenfunction on $\overline{M}$.
We say $f$ has the \emph{extension} \emph{property} iff there exists
an open neighborhood $\widehat{M}$ with $\overline{M}\subset\widehat{M}\subset\mathbf{M}$
and a Morse function $\widehat{f}:\widehat{M}\rightarrow\R$ such
that
\begin{equation}
\widehat{f}|{}_{\overline{M}}\equiv f.
\end{equation}

\end{defn}
Morse eigenfunctions with the extension property have the following
stronger extension property to Morse functions on $\mathbf{M}$.
\begin{lem}
\label{lem:extended_function}Let $f:\overline{M}\rightarrow\R$ be
a Morse-eigenfunction with the extension property, then there exists
a Morse function $\tilde{f}:\mathbf{M}\rightarrow\R$ that extends
$f$, i.e.

\begin{equation}
\tilde{f}|{}_{\overline{M}}\equiv f.
\end{equation}
\end{lem}
\begin{proof}
It follows from the extension property that there exist an open neighborhood
$W$ of $\overline{M}$ and a closed neighborhood $A$ of $\overline{M}$
as well as an extension $\widehat{f}$ such that
\begin{equation}
\mathscr{C}(\widehat{f}|{}_{W})\subset\mathring{A}\subset A\subset W\subset\mathbf{M},
\end{equation}
where $\mathscr{C}(\widehat{f}|{}_{W})\subset\mathring{A}$ follows
as $\widehat{f}$ has isolated critical points, so that $A$ can be
chosen to satisfy this inclusion. The existence of the extension $\tilde{f}$
to $\mathbf{M}$ follows immediately from Lemma 4.15 of \cite{Schwarz_MorseHomology}. \end{proof}
\begin{rem}
Determining general criteria for $M$ such that all Morse-eigenfunctions
have the extension property appears to be a non-trivial problem. In
\cite{MoNi57} local extendibility of eigenfunctions as solutions
under certain restrictive conditions are discussed. However, to apply
the extendibility of \cite{Schwarz_MorseHomology} one requires an
extension to a neighborhood of the whole domain. It seems therefore
reasonable to restrict to functions that allow for such an extension.
In particular, the extension property holds for Morse-eigenfunctions
on rectangular domains and on the disk, which follows from their explicitly
given eigenfunctions. Furthermore, it seems possible to extend the
treatment presented here to eigenfunctions which are weakly Morse
in the sense that they allow for degenerated saddle points. This would
require a more careful study of these cases beyond the standard theory
for Morse functions which we employ here.
\end{rem}
We use the notation $\Sd$ for the set of saddle points, which now
also includes the saddle points of $f$ on $\partial M$, and similarly
for $\Cr$. For extrema the sets $\Max$ and $\Min$ remain the same;
these cannot lie on the boundary since $f$ is an eigenfunction with
Dirichlet boundary conditions. In the following we define Neumann
domains and Neumann lines for manifolds with Dirichlet boundary (in
definition \ref{def:ND_using_CW_complex_Dirichlet_boundary} and its
preceding discussion) and prove that the results stated in proposition
\ref{prop:complementary_property_no_boundary} and theorem \ref{thm:topological_properties_no_boundary}
hold in a slightly different form in the boundary case (the analogues
are proposition \ref{prop:complementary_property_with_boundary} and
theorem \ref{thm-2}).

\subsection{Gradient flow in the boundary case}

Before discussing the structure of stable and unstable manifolds,
we introduce an adapted version of the gradient flow given in \eqref{eq:flow},
for manifolds with boundary. Let $f$ be a Morse eigenfunction on
$M$ with the extension property and $\tilde{f}$ the extended Morse
function on $\mathbf{M}$, as given in lemma \ref{lem:extended_function}.
Let $\tilde{\varphi}$ be the gradient flow of $\tilde{f}$ on $\mathbf{M}$,
as defined in \eqref{eq:flow}. We define the gradient flow of $f$
on $\overline{M}$ as

\begin{equation}
\begin{aligned} & \varphi:\mathbb{R}\times\,\overline{M}\rightarrow\overline{M},\\
 & \varphi(t,\,x)=\begin{cases}
\tilde{\varphi}\left(t,x\right) & \left(t\geq0\textrm{~and }\left\{ \tilde{\varphi}\left(t_{0},x\right)\right\} _{t_{0}\in\left[0,t\right]}\subset\overline{M}\right)~\textrm{or~}\\
 & \left(t<0\mbox{\textrm{~and }}\left\{ \tilde{\varphi}\left(t_{0},x\right)\right\} _{t_{0}\in\left[t,0\right]}\subset\overline{M}\right)\\
\tilde{\varphi}\left(t_{0},x\right) & \left(t>0\textrm{~with }t_{0}\in\left[0,t\right]~\textrm{being the minimal such that}~\tilde{\varphi}\left(t_{0},x\right)\in\partial M\right)~\textrm{or}~\\
 & \left(t<0\textrm{~with }t_{0}\in\left[t,0\right]~\textrm{being the maximal such that}~\tilde{\varphi}\left(t_{0},x\right)\in\partial M\right)
\end{cases}
\end{aligned}
\label{eq:flow-with-boundary}
\end{equation}

The flow above is along gradient lines and when a gradient line intersects
with the boundary, the flow is defined to stop at the intersection
point, or emanate from it, depending on the gradient direction. We
note that there is no essential need to use the extended function
when defining the flow above. However, using it somewhat simplifies
both the flow definition and the proofs to follow.

Note that the flow above is well-defined. For $t>0$, for example,
if the condition\\
$\left\{ \tilde{\varphi}\left(t_{0},x\right)\right\} _{t_{0}\in\left[0,t\right]}\subset\overline{M}$
fails to hold then there exists some $t_{1}\in\left(0,t\right]$ such
that $\tilde{\varphi}\left(t_{1},x\right)\notin\overline{M}$. By
definition of the flow $\tilde{\varphi}$ we have $\tilde{\varphi}\left(0,x\right)\in\overline{M}$.
This together with the continuity of the flow implies the existence
of $t_{0}\in\left[0,t\right]$ such that $\tilde{\varphi}\left(t_{0},x\right)\in\partial M$.
In particular there is a minimal $t_{0}$ which satisfies this, as
is required in \eqref{eq:flow-with-boundary}.

We claim that the gradient flow defined above does not depend on the
specific extended function $\tilde{f}$.
\begin{lem}
Let $f$ be a Morse function on $M$. Let $\tilde{f_{1}}$ and $\tilde{f_{2}}$
be extensions of $f$ from $M$ to $\mathbf{M}$, $\tilde{\varphi}_{1},\tilde{\varphi}_{2}$
the corresponding gradient flows and $\varphi_{1},\varphi_{2}$ the
corresponding flows induced by \eqref{eq:flow-with-boundary}. Then
$\varphi_{1}=\varphi_{2}$.\end{lem}
\begin{proof}
Assume by contradiction that there exists $t\in\R,~x\in\overline{M}$
such that $\varphi_{1}\left(t,x\right)\neq\varphi_{2}\left(t,x\right)$.
This implies that $t\neq0$, as $\varphi_{1}\left(0,x\right)=\varphi_{2}\left(0,x\right)=x$,
by definition. We may assume without loss of generality that $t>0$.
If $\left\{ \tilde{\varphi}_{1}\left(t_{0},x\right)\right\} _{t_{0}\in\left[0,t\right]}\subset\overline{M}$
then we get by the flow definition \eqref{eq:flow} that $\forall t_{0}\in\left[0,t\right],~~\tilde{\varphi}_{1}\left(t_{0},x\right)=\tilde{\varphi}_{2}\left(t_{0},x\right)$,
since $\nabla\tilde{f}_{1}|_{\overline{M}}=\nabla\tilde{f}_{2}|_{\overline{M}}$.
Hence $\tilde{\varphi}_{1}\left(t,x\right)=\tilde{\varphi}_{2}\left(t,x\right)$,
contradicting the assumption. Therefore, there exists a minimal $t_{1}\in\left[0,t\right]$
such that $\tilde{\varphi}_{1}\left(t_{1},x\right)\in\partial M$
and a minimal $t_{2}\in\left[0,t\right]$ such that $\tilde{\varphi}_{2}\left(t_{2},x\right)\in\partial M$.
This means that $\forall t_{0}\in\left[0,t_{1}\right)$ $\tilde{\varphi}_{1}\left(t_{0},x\right)\in M$
and implies $\forall t_{0}\in\left[0,t_{1}\right)~\tilde{\varphi}_{1}\left(t_{0},x\right)=\tilde{\varphi}_{2}\left(t_{0},x\right)$.
By the continuity of the flow $\tilde{\varphi}_{2}\left(t_{1},x\right)=\tilde{\varphi}_{1}\left(t_{1},x\right)\in\partial M$,
so that $t_{1}\in\left[0,t\right]$ is the minimal such that $\tilde{\varphi}_{2}\left(t_{1},x\right)\in\partial M$.
Therefore, $t_{1}=t_{2}$ and $\tilde{\varphi}_{1}\left(t_{1},x\right)=\tilde{\varphi}_{2}\left(t_{2},x\right)$,
contradicting the assumption.
\end{proof}
\noindent This definition of the flow allows to define the stable
and unstable manifolds similarly to the non-boundary case by
\begin{equation}
\begin{aligned}W^{s}(x) & =\{y\in\overline{M}\big|\lim_{t\rightarrow\infty}\varphi(t,y)=x\}\\
W^{u}(x) & =\{y\in\overline{M}\big|\lim_{t\rightarrow-\infty}\varphi(t,y)=x\}.
\end{aligned}
\end{equation}
Note that the above is defined not only for critical points, but also
for $x\in\partial M$. In order to prove properties of the stable
and unstable manifolds, we need the following lemma.
\begin{lem}
{[}follows from Proposition 3.18 in \cite{BanHur_MorseHomology04}{]}\label{lem:function_decreases_along_flow_line}
Let $M$ be a smooth Riemannian manifold with or without boundary.
Let $x\in M$ and $t\in\R$ be such that $\varphi(t,x)\notin\partial M$.
Then
\begin{equation}
\frac{\textrm{d}}{\textrm{d}t}f\left(\varphi(t,x)\right)=-\left\Vert \left(\nabla f\right)\left(\varphi(t,x)\right)\right\Vert ^{2}\leq0.
\end{equation}

\end{lem}
The next lemma is analogous to lemma \ref{lem:stable_mnfold_are_open_sets},
for the case with boundary.
\begin{lem}
\noindent \label{lem:stable_mnfold_are_open_sets-bdry-case}Let $M$
be a manifold with boundary as above. Let $f$ be a Morse-eigenfunction
on $M$ with the extension property. Let $p\in\Cr$ and let $\lambda_{p}$
be the Morse index of $f$ at $p$. The intersection of the stable
(unstable) manifold with the interior of $M$, $W^{s}(p)\cap\mathrm{int}M$
($W^{u}(p)\cap\mathrm{int}M$) is an open simply connected set of
dimension $2-\lambda_{p}$ ($\lambda_{p}$).\end{lem}
\begin{proof}
We prove the statement above separately for the three different cases,
$\lambda_{p}=2,1,0$ and only for $W^{s}(p)\cap\mathrm{int}M$. The
first case ($\lambda_{p}=2$) holds due to the fact that extrema of
eigenfunctions with Dirichlet boundary conditions belong to the interior
of $M$, therefore lemma \ref{lem:stable_mnfold_are_open_sets} applies
to this case and $W^{s}\left(p\right)\cap\textrm{Int}M=\left\{ p\right\} $.\\
 For the second and third cases ($\lambda_{p}=1,0$), we use the extension
property and extend the eigenfunction to a smooth Morse function $\tilde{f}$
on $\mathbf{M}$. We now prove the $\lambda_{p}=1$ case. Denote the
stable manifold of $p$ with respect to $\tilde{f}$ by $\widetilde{W}^{s}\left(p\right)$.
This stable manifold is defined with respect to the standard gradient
flow on $\mathbf{M}$, which we denote by $\tilde{\varphi}:\R\times\mathbf{M}\rightarrow\mathbf{M}.$
Note that the flows $\tilde{\varphi}$ and $\varphi$ coincide in
the interior of $M$. By lemma \ref{lem:stable_mnfold_are_open_sets}
the stable manifold $\widetilde{W}^{s}\left(p\right)$ is an embedded
open interval in $\mathbf{M}$. If its intersection with the interior
of $M$ is connected, this would imply the claim. If $\widetilde{W}^{s}(p)\cap\mathrm{int}M$
is not connected, then one of the integral curves in $\widetilde{W}^{s}\left(p\right)$
(i.e., $\left\{ \tilde{\varphi}\left(t,x\right)\right\} _{t\in\R}$
which is contained in $\widetilde{W}^{s}\left(p\right)$) must intersect
$\partial M$ at least at two points. Denote these points on the boundary
by $x_{1},x_{2}$ and let $t_{2}>0$ such that $x_{2}=\tilde{\varphi}\left(t_{2,}x_{1}\right)$.
According to lemma \ref{lem:function_decreases_along_flow_line},
the values of $\tilde{f}$ decrease monotonically along the flow line,
$\left\{ \tilde{\varphi}\left(x_{1};t\right)\right\} _{0\leq t\leq t_{2}}$.
As $\tilde{f}$ vanishes at both $x_{1}$ and $x_{2}$, we conclude
from lemma \ref{lem:function_decreases_along_flow_line} that $\frac{\textrm{d}}{\textrm{d}t}f\left(\tilde{\varphi}\left(x;t\right)\right)=-\left\Vert \left(\nabla f\right)\left(\tilde{\varphi}\left(x_{1},t\right)\right)\right\Vert ^{2}=0~~~\forall0\leq t\leq t_{2}$.
The existence of non-isolated critical points contradicts $\tilde{f}$
being a Morse function and finishes the proof of the second claim.\\
 To prove the third case ($\lambda_{p}=0$), let $p\in\Min$. Then
the stable manifold $\widetilde{W}^{s}(p)$ is an embedded open disk
in $\mathbf{M}$ and in particular simply connected. Since $M$ is
simply connected, we may use the same argument as in the proof of
theorem \ref{thm:topological_properties_no_boundary}\eqref{enu:thm-no-bdry-simply-connected}
to conclude that $\widetilde{W}^{s}(p)\cap\mathrm{int}M$ is simply
connected if it is path connected. The set $\widetilde{W}^{s}(p)\cap\mathrm{int}M$
is indeed path connected as $x,y\in\widetilde{W}^{s}(p)\cap\mathrm{int}M$
are connected by the path $\overline{\left\{ \tilde{\varphi}\left(t,x\right)\right\} _{t\geq0}}\cup\overline{\left\{ \tilde{\varphi}\left(t,y\right)\right\} _{t\geq0}}$.
This holds since the gradient flow lines $\left\{ \tilde{\varphi}\left(t,x\right)\right\} _{t\geq0}$,
$\left\{ \tilde{\varphi}\left(t,y\right)\right\} _{t\geq0}$ are fully
contained in $\textrm{int}M$ (as is explained in the previous paragraph)
and $p$ is a common point in the closures of these ($\lim_{t\rightarrow-\infty}\tilde{\varphi}\left(t,x\right)=\lim_{t\rightarrow-\infty}\tilde{\varphi}\left(t,y\right)=p$).
$\widetilde{W}^{s}(p)\cap\mathrm{int}M$ is therefore a simply connected
open set and hence homeomorphic to a two-dimensional disk.
\end{proof}

\subsection{Neumann domains in the boundary case}

\noindent Let $M$ be a connected compact 2-manifold with boundary
as described in the beginning of this section and consider a Morse-eigenfunction
$f$ with the extension property, which obeys Dirichlet boundary conditions
on $\partial M$. The definition of the set of Neumann lines, $\Nd$,
in this case is unaltered and is still given by definition \ref{def:ND_using_saddles}.
The definition of Neumann domains, however, should be modified as
follows.
\begin{defn}
\label{def:ND_using_CW_complex_Dirichlet_boundary}A \emph{Neumann
domain} of $f$ as above is a connected component of one of the following
sets
\begin{enumerate}
\item $\Omega_{p,q}(f)=W^{s}\left(p\right)\cap W^{u}\left(q\right),$ where
$p\in\Max,\,q\in\Min$
\item $\Omega_{p,\circ}(f)=W^{s}\left(p\right)\cap\left(\bigcup_{y\in\partial M\backslash\Sd}W^{u}\left(y\right)\right),$
where $p\in\Min$
\item $\Omega_{\circ,q}(f)=\left(\bigcup_{y\in\partial M\backslash\Sd}W^{s}\left(y\right)\right)\cap W^{u}\left(q\right),$
where $q\in\Max$
\end{enumerate}
Neumann domains of type (1) are called \emph{inner Neumann domains}
and those of types (2) and (3) are called \emph{boundary Neumann domains}.
\end{defn}
\noindent Similarly to lemma \ref{lem:decomp_to_stble_mnfolds}, we
have an analogue decomposition of $M$.
\begin{lem}
\label{lem:decomp_to_stble_mnfolds-bdry_case} Let $M$ and $f$ be
as above, then we have the following disjoint decompositions
\begin{equation}
\overline{M}=\left\{ \bigsqcup_{x\in\Cr}W^{u}\left(x\right)\right\} \bigsqcup\left\{ \bigsqcup_{y\in\partial M\backslash\Sd}W^{u}\left(y\right)\right\} .
\end{equation}
Similarly,
\begin{equation}
\overline{M}=\left\{ \bigsqcup_{x\in\Cr}W^{s}\left(x\right)\right\} \bigsqcup\left\{ \bigsqcup_{y\in\partial M\backslash\Sd}W^{s}\left(y\right)\right\} .
\end{equation}
\end{lem}
\begin{proof}
Both decompositions follow as each point belongs to a unique (un)stable
manifold and there are no extremal points on a Dirichlet boundary.
\end{proof}
\noindent The following lemma is the analogue of lemma \ref{lem:flow_lim_is_critic_pt}.
\begin{lem}
\label{lem:flow_lim_is_critic_pt_or_bdry}Let $x\in M$. Then both
limits $\lim_{t\rightarrow\infty}\varphi\left(t,x\right)$ and $\lim_{t\rightarrow-\infty}\varphi\left(t,x\right)$
exist and each is either a critical point of $f$ or an element of
$\partial M$, i.e., $\lim_{t\rightarrow\pm\infty}\varphi\left(t,x\right)\in\Cr\cup\partial M$..\end{lem}
\begin{proof}
Let $x\in M$. If $\left\{ \varphi\left(t,x\right)\right\} _{t\in\R}\cap\partial M=\emptyset$
then lemma \ref{lem:flow_lim_is_critic_pt} applies and we get that
the limits $\lim_{t\rightarrow\infty}\varphi_{x}\left(t,x\right)$
and $\lim_{t\rightarrow-\infty}\varphi\left(t,x\right)$ exist and
both belong to $\Cr$ (and it might be that any of these limits belongs
to the boundary, $\partial M$). Otherwise, there exists $t_{0}\in\R$
such that $\varphi\left(t_{0},x\right)=y$ and $y\in\partial M$.
Assume without loss of generality that $t_{0}<0$ . Note that due
to the reversibility of the flow, the gradient cannot vanish at $y$,
i.e., $\left.\nabla f\right|_{y}\neq0$. We conclude that $y$ cannot
be a corner of the boundary as this would imply $y\in\Sd$. Therefore
$y$ belongs to the smooth part of the boundary and $\left.\nabla f\right|_{y}$
is orthogonal to the boundary, which is a level set of $f$ . By the
definition of the flow, we get that $\forall t<t_{0},~\varphi\left(t,x\right)=y$
and therefore $\lim_{t\rightarrow-\infty}\varphi\left(t,x\right)=y$.
\end{proof}
\noindent The next two lemmata are the analogues of lemmata \ref{lem:boundary_of_stble_mnfld_of_saddle}
and \ref{lem:Neumann_nodes_contain_all_extrema}.
\begin{lem}
\label{lem:boundary_of_stble_mnfld_of_saddle-for_mnfold_with_bdry}
Let $r\in\Sd.$ Then $q\in\overline{W^{s}\left(r\right)}\backslash W^{s}\left(r\right)$
if and only if $q\in\Cr$ and $W^{s}\left(r\right)\cap W^{u}\left(q\right)\neq\emptyset$.
Similarly, $p\in\overline{W^{u}\left(r\right)}\backslash W^{u}\left(r\right)$
if and only if $p\in\Cr$ and $W^{u}\left(r\right)\cap W^{s}\left(p\right)\neq\emptyset$. \end{lem}
\begin{proof}
The proof of direction ($\Leftarrow$) is identical to that of lemma
\ref{lem:boundary_of_stble_mnfld_of_saddle}. The proof of the other
direction is only slightly modified (using lemmata \ref{lem:stable_mnfold_are_open_sets-bdry-case},
\ref{lem:flow_lim_is_critic_pt_or_bdry} which are analogous to lemmata
\ref{lem:stable_mnfold_are_open_sets}, \ref{lem:flow_lim_is_critic_pt})
and is not repeated. The only element of proof which we do mention
here concerns points on the boundary $\partial M$. Let $y\in\partial M\cap\overline{W^{s}\left(r\right)}$.
Then we have that $y\in W^{s}\left(r\right)$ if and only if $y\notin\Cr$.
In particular, if $y\in\partial M$ then $y\in\overline{W^{s}\left(r\right)}\backslash W^{s}\left(r\right)$
if and only if $y\in\Cr$.
\end{proof}
\noindent The main content of the following lemma (similarly to lemma
\ref{lem:Neumann_nodes_contain_all_extrema}) is that all extremal
points belong to the set of Neumann lines, $\Nd$. The proof is somewhat
more involved here than the proof of the analogous lemma in the non-boundary
case. Intuitively, this can be understood as following. Given an eigenfunction,
$f$, on the manifold with boundary $M$, we extend it to a Morse
function, $\tilde{f}$, on the manifold $\mathbf{M}$. Yet, it might
occur that $N(f)|_{M}\subsetneq N(\tilde{f})|_{M}$. Namely, some
Neumann lines of $\tilde{f}$ might be absent from those of $f$ even
if originally they had a non-empty intersection with $M$. This happens
exactly when a saddle point of $\tilde{f}$ lies outside $M$, but
has a stable or unstable manifold which intersects with $M$. In particular,
having less Neumann lines means that it is harder to guarantee that
all extremal points belong to the Neumann lines, in the boundary case.
Hence the difference in the complexity of the proofs.
\begin{lem}
\label{lem:Neumann_nodes_contain_all_extrema_bdry_case}If $\Nd\neq\emptyset$
then
\begin{equation}
\Nd=\left\{ \bigcup_{{r\in\Sd}}W^{s}(r)\cup W^{u}(r)\right\} \bigsqcup\Xt.
\end{equation}
\end{lem}
\begin{proof}
This proof partly follows the lines of the one for lemma \ref{lem:Neumann_nodes_contain_all_extrema},
if we replace lemmata \ref{lem:stable_mnfold_are_open_sets}, \ref{lem:flow_lim_is_critic_pt}
and \ref{lem:boundary_of_stble_mnfld_of_saddle} by the analogous
lemmata \ref{lem:stable_mnfold_are_open_sets-bdry-case}, \ref{lem:flow_lim_is_critic_pt_or_bdry}
and \ref{lem:boundary_of_stble_mnfld_of_saddle-for_mnfold_with_bdry}.
However, this proof deviates at some point and additional arguments
are supplied. We observe that
\begin{equation}
\begin{aligned}\Nd & \subseteq\left\{ \bigcup_{{r\in\Sd}}W^{s}(r)\cup W^{u}(r)\right\} \bigcup\Cr\\
 & =\left\{ \bigcup_{{r\in\Sd}}W^{s}(r)\cup W^{u}(r)\right\} \bigsqcup\Xt,
\end{aligned}
\end{equation}
where the first line is a deduction from lemma \ref{lem:boundary_of_stble_mnfld_of_saddle-for_mnfold_with_bdry}
and the second equality holds as $\Cr\backslash\Xt=\Sd\subset\bigcup_{{r\in\Sd}}W^{s}(r)$.

We proceed to show that the relation above is an exact equality. Let
$q$ be a maximum of $f$. We show that $\exists r\in\Sd$ such that
$q\in\overline{W^{s}(r)}$. Similar arguments can be used to show
that if $p$ is a minimum of $f$ then $\exists r\in\Sd$ such that
$p\in\overline{W^{u}(r)}$ and in combination this proves the lemma.
Examine $W^{u}\left(q\right)$. If $\partial M\subset W^{u}\left(q\right)$
we conclude $\Sd=\emptyset$ and therefore also $\Nd=\emptyset$.
This conclusion owes to $W^{u}\left(q\right)\cap\textrm{int}M$ being
simply connected (lemma \ref{lem:stable_mnfold_are_open_sets-bdry-case}),
so that its boundary equals $\partial M$ and therefore $M=W^{u}\left(q\right)$,
which implies $\Sd=\emptyset$.

We now consider the case $W^{u}(q)\cap\partial M=\emptyset$. In particular,
the unstable manifold of $q$ is contained in $M$ and therefore coincides
with that of the extension $\widetilde{W}^{u}(q)$, $W^{u}(q)=\widetilde{W}^{u}(q)$.
For $\mathbf{M}$ we can proceed as in the proof of lemma \ref{lem:Neumann_nodes_contain_all_extrema}
and deduce the existence of a saddle point $r\in\mathscr{S}(\tilde{f})$
with $q\in\overline{\widetilde{W}^{s}(r)}$. In particular, there
is a gradient flow line connecting $q$ and $r$. This gradient flow
line is contained in $W^{u}(q)$ by definition and in turn it is contained
in $M$. Therefore $r\in\overline{M}$ is the desired saddle point.

If $W^{u}(q)\cap\partial M\neq\emptyset$, we consider $\partial W^{u}\left(q\right)$,
taking the boundary with respect to the topology of $\overline{M}$.
$\partial W^{u}\left(q\right)$ is not empty as we have shown that
$W^{u}\left(q\right)\subsetneq\overline{M}$. $\partial W^{u}\left(q\right)$
is a compact set so that $f$ attains a maximum on it, at some $r\in\partial W^{u}\left(q\right)$.
We first assume that $r\notin\partial M$ and show that $\left.\nabla f\right|_{r}=0$.
First, note that $\partial W^{u}\left(q\right)$ is a one-dimensional
curve being part of the boundary of an embedded two-dimensional disk,
by lemma \ref{lem:stable_mnfold_are_open_sets-bdry-case}. If $\partial W^{u}\left(q\right)$
is smooth at $r$, then, since $r$ is a local maximum of $\left.f\right|_{\partial W^{u}\left(q\right)}$,
we get that $\left.\nabla f\right|_{r}$ is orthogonal to the tangent
of $\partial W^{u}\left(q\right)$ at $r$. If $\left.\nabla f\right|_{r}\neq0$
this implies that $\left\{ \varphi\left(t,r\right)\right\} _{t\in\R}\cap W^{u}\left(q\right)\neq\emptyset$
and therefore $\lim_{t\rightarrow-\infty}\varphi\left(t,r\right)=q$,
contradicting $r\in\partial W^{u}\left(q\right)$. If $\partial W^{u}\left(q\right)$
has a corner at $r,$ $\left.\nabla f\right|_{r}$ is orthogonal to
both right and left tangents of $\partial W^{u}\left(q\right)$ at
$r$ and once again $\left.\nabla f\right|_{r}=0$. If $r\in\partial M$
then $\left.\nabla f\right|_{r}$ is orthogonal to both $\partial W^{u}\left(q\right)$
and $\partial M$ (as the latter is a level set) so that $\left.\nabla f\right|_{r}=0.$
We conclude that $r\in\Cr$. Obviously, $r$ cannot be a minimum and
it also cannot be a maximum since this will yield $W^{u}\left(q\right)\cap W^{u}\left(r\right)\neq\emptyset$.
Therefore, $r\in\Sd$. $W^{u}\left(r\right)$ cannot intersect $W^{u}\left(q\right)$
and therefore it is tangential to $\partial W^{u}\left(q\right)$
at $r.$ From the local structure of $\Nd$ near $r$, as given in
lemma \ref{lem:local_Neumann_set_near_saddle}\eqref{enu:lem_local_Neumann_near_saddle-1}
we deduce that $W^{s}\left(r\right)$ is transversal to $\partial W^{u}\left(q\right)$.
Therefore, $W^{s}\left(r\right)\cap W^{u}\left(q\right)\neq\emptyset$
and we conclude from lemma \ref{lem:boundary_of_stble_mnfld_of_saddle-for_mnfold_with_bdry}
that $q\in\overline{W^{s}(r)}$ as desired.

It is left to consider the case $r\in\partial M$. If there is some
other isolated maximum on $\partial W^{u}\left(q\right)$, we pick
it as $r$ and proceed as before. Otherwise, either $\left.f\right|_{\partial W^{u}\left(q\right)}\equiv0$
or $\left.f\right|_{\partial W^{u}\left(q\right)}<0$. The case $\left.f\right|_{\partial W^{u}\left(q\right)}\equiv0$
can be ruled out as then for $x\in\partial W^{u}\left(q\right)$ we
have $\nabla f\left(x\right)\bot\partial W^{u}\left(q\right)$ which
either contradicts $x\in\partial W^{u}\left(q\right)$ or implies
$x\in\Sd$. We therefore get that $\partial W^{u}\left(q\right)\subset\Sd$
which contradicts $f$ being a Morse function. We conclude that $f\left(x\right)<0~~\forall x\in\partial W^{u}\left(q\right)$.
The strict negative sign of $f$ on the boundary of the unstable manifold
in combination with $f\left(q\right)>0$ implies
\begin{equation}
W^{u}\left(q\right)\cap f^{-1}\left(0\right)\neq\emptyset.
\end{equation}
This set cannot contain a closed nodal line in the interior of $M$
for the following reason. In case the maximum is contained in the
corresponding nodal domain, this would imply that gradient flow lines
connecting the maximum and the boundary, $\partial M$, attain the
zero value twice, which contradicts the monotonicity of $f$ along
gradient flow lines (lemma \ref{lem:function_decreases_along_flow_line}).
Assuming that the corresponding nodal domain does not contain the
maximum, then it contains another extremal point of the eigenfunction.
But no extremal point other than $q$ is an element of the unstable
manifold of $q$, by definition. We therefore deduce that the nodal
set $f^{-1}\left(0\right)$ has a non-empty intersection with the
boundary,
\begin{equation}
f^{-1}\left(0\right)\cap\partial M\cap\overline{W^{u}\left(q\right)}\neq\emptyset,
\end{equation}
as nodal lines are either closed or end at the boundary. We pick $s\in f^{-1}\left(0\right)\cap\partial M\cap\overline{W^{u}\left(q\right)}$
and claim that this is the required saddle point. The fact that it
is a saddle point follows immediately since nodal lines intersect
the boundary of $M$ at saddle points of $f$ on the boundary. The
local structure of $\Nd$ and $f^{-1}\left(0\right)$ in the vicinity
of $s$, as described in lemma \ref{lem:local_Neumann_set_near_saddle}\eqref{enu:lem_local_Neumann_near_saddle-2}
implies that $W^{s}\left(s\right)\cap W^{u}\left(q\right)\neq\emptyset$
and we conclude from lemma \ref{lem:boundary_of_stble_mnfld_of_saddle-for_mnfold_with_bdry}
that $q\in\overline{W^{s}(s)}$ as desired.
\end{proof}
\noindent Finally, the analogue of proposition \ref{prop:complementary_property_no_boundary}
is
\begin{prop}
\noindent \label{prop:complementary_property_with_boundary} If $\Nd\neq\emptyset$
then the following disjoint decomposition of the manifold holds.
\begin{equation}
\begin{aligned}\overline{M}= & \left\{ \bigsqcup_{\begin{array}{c}
{\textrm{\ensuremath{p\in\Min}}\atop {\textrm{\ensuremath{q\in\Max}}}}\end{array}}\!\!\!\!\!\!\Omega_{p,q}(f)\right\} \bigsqcup\left\{ \bigsqcup_{\begin{array}{c}
{\textrm{\ensuremath{p\in\Min}}\atop {}}\end{array}}\!\!\!\!\!\!\Omega_{p,\circ}(f)\right\} \bigsqcup\left\{ \bigsqcup_{\begin{array}{c}
{{\textrm{\ensuremath{q\in\Max}}}\atop }\end{array}}\!\!\!\!\!\!\Omega_{\circ,q}(f)\right\} \,\,\bigsqcup\,\,N\left(f\right).\end{aligned}
\end{equation}
\end{prop}
\begin{proof}
\noindent Note the following disjoint decomposition of the manifold
\begin{equation}
\begin{aligned}M & =\left\{ \bigsqcup_{{{p\in\Min}\atop {q\in\Max}}}\!\!\!\left[W^{s}\left(p\right)\cap W^{u}\left(q\right)\right]\right\} \\
 & \quad\bigsqcup\left\{ \bigsqcup_{{{p\in\Min}\atop {}}}\left[W^{s}\left(p\right)\cap\left(\bigcup_{{y\in\partial M\backslash\Sd\atop {}}}W^{u}\left(y\right)\right)\right]\right\} \\
 & \quad\bigsqcup\left\{ \bigsqcup_{{\ensuremath{q\in\Max}\atop {}}}\,\,\,\!\!\!\!\left[\left(\bigcup_{{y\in\partial M\backslash\Sd\atop {}}}W^{s}\left(y\right)\right)\cap W^{u}\left(q\right)\right]\right\} \\
 & \quad\bigsqcup\left\{ \bigsqcup_{{{r\in\Sd}\atop }}\left[W^{s}(r)\cup W^{u}(r)\right]\right\} \bigsqcup\Xt,
\end{aligned}
\end{equation}
whose validity follows from separation into cases (see beginning of
the proof of proposition \ref{prop:complementary_property_no_boundary})
together with $f$ having no extremal points on the Dirichlet boundary.
The proposition now follows from definition \ref{def:ND_using_CW_complex_Dirichlet_boundary}
and lemma \ref{lem:Neumann_nodes_contain_all_extrema_bdry_case}.
\end{proof}
\noindent The following structural theorem provides the same results
as theorem \ref{thm:topological_properties_no_boundary} for inner
Neumann domains and analogous results for the boundary Neumann domains.
\begin{thm}
\noindent \label{thm-2} Let $M$ be a simply connected, compact subset
of a compact, closed 2-dimensional smooth manifold $\mathbf{M}$.
Let $\partial M$ be piecewise smooth , all of whose angles are strictly
positive and let $g$ be the metric on $M$. Let $f$ be a Morse eigenfunction
of $-\De_{g}$ with the extension property, which obeys Dirichlet
boundary conditions and is such that $\Sd\neq\emptyset$.\\
 The following holds.
\begin{enumerate}
\item $\Cr\subset\Nd$ \label{enu:thm2-referring_to_thm1_1}
\end{enumerate}

\subparagraph*{\uline{Inner Neumann domains}\global\long\def\theenumi{\roman{enumi}}
}
\begin{enumerate}[resume]
\item Claims \eqref{enu:thm-no-bdry-extremal-points} and \eqref{enu:thm-no-bdry-simply-connected}-\eqref{enu:thm-no-bdry-level-sets-3}
of theorem \ref{thm:topological_properties_no_boundary} hold for
all inner Neumann domains of $f$.\label{enu:thm2-referring_to_thm1_2}~
\end{enumerate}

\subparagraph*{\uline{Boundary Neumann domains}\global\long\def\theenumi{\roman{enumi}}
}

Let $p\in\Min$ and let $\Omega$ be a connected component of $\Omega_{p,\circ}(f)$.
Then
\begin{enumerate}[resume]
\item $\partial\Omega\cap\Xt=\{p\}$\label{enu:thm2-sec-1}.
\item $\Omega$ is simply connected. \label{enu:thm2-simply-connected}
\item $\left.f\right|_{\Omega\backslash\partial M}<0$ and therefore $\Omega\cap f^{-1}(0)=\Omega\cap\partial M$\label{enu:thm2-sec-3}.
\end{enumerate}
Analogous claims hold for boundary Neumann domains of maxima.\end{thm}
\begin{proof}
Claims \eqref{enu:thm-no-bdry-critical-points},\eqref{enu:thm-no-bdry-extremal-points}
here are proven identically as in theorem \ref{thm:topological_properties_no_boundary}
(with lemma \ref{lem:Neumann_nodes_contain_all_extrema_bdry_case}
as the analogue of lemma \ref{lem:Neumann_nodes_contain_all_extrema}).
Claim \eqref{enu:thm2-simply-connected} for boundary Neumann domains
is also proven as its analogue (claim \eqref{enu:thm-no-bdry-simply-connected})
in theorem \ref{thm:topological_properties_no_boundary}.\\
 {[}Proof of \eqref{enu:thm2-sec-3}{]} Assume by contradiction that
there exists $x\in\Omega\backslash\partial M$ such that $f\left(x\right)\geq0$.
Consider the set $\left\{ \varphi\left(t,x\right)\right\} _{t>0}$.
By definition of $\Omega$, $\lim_{t\rightarrow-\infty}\varphi\left(t,x\right)\in\partial M$
and therefore $\lim_{t\rightarrow-\infty}f\left(\varphi\left(t,x\right)\right)=0$.
Lemma \ref{lem:function_decreases_along_flow_line} states that $f$
cannot increase along gradient flow lines and therefore $f\left(\varphi\left(t,x\right)\right)=0~\forall t\leq0$
and we conclude that $\frac{\textrm{d}}{\textrm{d}t}f\left(\varphi\left(t,x\right)\right)=-\left\Vert \left(\nabla f\right)\left(\varphi\left(t,x\right)\right)\right\Vert ^{2}=0$
for $t\leq0\,$. We get a set, $\left\{ \varphi\left(t,x\right)\right\} _{t\leq0}$,
of non-isolated critical points of $f$, in contradiction to $f$
being a Morse function. We therefore have $\left.f\right|_{\Omega\backslash\partial M}<0$
and conclude $\Omega\cap f^{-1}(0)=\Omega\cap\partial M$. \\
 {[}Proof of \eqref{enu:thm2-sec-1}{]} Proving that $\partial\Omega\cap\Min=\left\{ p\right\} $
is done similarly to the analogue claim in theorem \ref{thm:topological_properties_no_boundary}
(claim \eqref{enu:thm-no-bdry-extremal-points}). It is then left
to show that $\partial\Omega\cap\Max=\emptyset$. As $\left.f\right|_{\Omega\backslash\partial M}<0$
by claim \eqref{enu:thm2-sec-3}, we conclude $\left.f\right|_{\overline{\Omega}}=\left.f\right|_{\overline{\Omega\backslash\partial M}}\leq0$.
Since $f$ is positive at its maxima points, being an eigenfunction,
we conclude that $\overline{\Omega}\cap\Max=\emptyset$ as required.
\end{proof}
\noindent In addition to the theorem above, one can make some straightforward
observations regarding Neumann lines which intersect with the boundary.
We first note that every critical point on the boundary is a saddle
point. In the vicinity of those saddle points the eigenfunction behaves
locally as it does in the neighborhood of a nodal line intersection
(which is like a harmonic polynomial, \cite{CH76}). The Neumann line
structure near those boundary saddle points can be deduced from lemma
\ref{lem:local_Neumann_set_near_saddle}. More explicitly, there are
two types of saddle points at the boundary. Those that are located
at corners of the manifold (each corner has such a saddle point) and
those which are located at a point on the smooth part of the boundary.
The former have a single Neumann line to which they are connected.
The latter are connected to two perpendicular Neumann lines and to
a nodal line which lies in between.
\begin{rem}
We briefly comment on surfaces with Neumann boundary conditions. The
flow in this case may be defined as in \eqref{eq:flow} and the boundary
of such a surface is naturally a union of gradient flow lines. Neumann
domains and Neumann lines are defined as for non-boundary surfaces
(definitions \ref{def:ND_using_CW_complex} and \ref{def:ND_using_saddles}).
Proposition \ref{prop:complementary_property_no_boundary} and all
claims of theorem \eqref{enu:thm-no-bdry-critical-points}, except
\eqref{enu:thm-no-bdry-Morse-Smale} hold.
\end{rem}

\section{Geometric and Spectral properties of Neumann domains\label{sec:geom_spec_properties}}

\noindent We have discussed so far the topological structure of Neumann
domains. We proceed by pointing out a connection between the aforementioned
results and geometric and spectral properties of Neumann domains.

\subsection{On the outer radius of Neumann domains\label{chap outer radius}}

The volume of nodal domains of an eigenfunction is bounded from below
in terms of the eigenvalue (by the Faber-Krahn inequality, \cite{Faber_23,Krahn_ma25}).
A lower bound on the Neumann domain volume does not exist. In particular,
there are continuous families of eigenfunctions of a fixed multiple
eigenvalue on the standard 2-torus which possess Neumann domains whose
volumes go to zero (cf.~\cite{McDFul_ptrs13}). However, the next
theorem provides a lower bound on the number of ``large-diameter''
Neumann domains.
\begin{defn}
\noindent Let $(M,g)$ be a two-dimensional Riemannian manifold without
boundary and $\Omega$ an open simply-connected subset of $M$. Let
$B_{r}(x_{0})$ denote a geodesic ball of radius $r$ around $x_{0}$.
Then we define the outer radius $R(\Omega)$, by
\begin{equation}
R(\Omega)=\inf\{r>0\,:\,\exists x_{0}\in M\,:\,\Omega\subset B_{r}(x_{0})\}.
\end{equation}
\end{defn}
\begin{thm}
\label{thm-lower_bound_on_outer_radius} Let $(M,g)$ be a two-dimensional
Riemannian manifold without boundary and $f$ a Morse-eigenfunction
of the Laplace-Beltrami operator with eigenvalue $\la$. Let $\nu$
denote the number of nodal domains of $f$. Then there exists a real
positive constant $C$ only depending on the metric $g$ such that
for at least $\left\lceil \nicefrac{\nu}{2}\right\rceil $ Neumann
domains $\{\Omega_{i}\}_{1\leq i\leq\left\lceil \nicefrac{\nu}{2}\right\rceil }$
of $f$
\begin{equation}
R(\Omega_{i})\geq C\la^{-1/2}.\label{eq:lower_bound_on_outer_radius}
\end{equation}
\end{thm}
\begin{proof}
This theorem follows from the structure of Neumann domains (theorem
\ref{thm:topological_properties_no_boundary}) and the bound on the
inner radius of nodal domains by Mangoubi \cite{Mangoubi_cmb08} (see
also \cite{NazPolSod_ajm05}). Each nodal domain $D$ of $f$ has
a global extremum of $\left.f\right|_{D}$. Each of these $\nu$ extrema
are members of the boundary of at least one Neumann domain each, by
theorem \ref{thm:topological_properties_no_boundary}. Let $q$ be
one of those maxima and $D_{q}$, the corresponding nodal domain of
$f$. By section 3 of \cite{Mangoubi_cmb08} there is a positive constant
$C'$ independent of $\la$ and a geodesic ball $B_{\nicefrac{C'}{\sqrt{\lambda}}}(q)$
such that $B_{\nicefrac{C'}{\sqrt{\lambda}}}(q)\subset D_{q}$. Let
$\Omega$ be a Neumann domain such that $q\in\partial\overline{\Omega}$
and let $p$ be the unique minimum on $\partial\Omega$ (not necessarily
a global minimum). Let $\gamma\left(p,q\right)$ be the geodesic ray
between $p$ and $q$ and $d\left(p,q\right)$ its length. As $f\left(p\right)<0,\,f\left(q\right)>0$
and by continuity of $f$, there exists $x\in\gamma\left(p,q\right)$
such that $f\left(x\right)=0$. We therefore get that
\begin{equation}
R\left(\Omega\right)\geq\frac{1}{2}d\left(p,q\right)\geq\frac{1}{2}d\left(x,q\right)\geq\frac{1}{2}C'\la^{-1/2}.\label{eq:bound_on_outer_radius}
\end{equation}
The argument above holds for each extremum which is global within
its nodal domain. Yet, as a Neumann domain may contain two such extrema
on its boundary, we deduce that \eqref{eq:bound_on_outer_radius}
holds for at least $\left\lceil \nicefrac{\nu}{2}\right\rceil $ of
the Neumann domains.\end{proof}
\begin{rem}
The number of Neumann domains for which the theorem holds, $\left\lceil \nicefrac{\nu}{2}\right\rceil $,
may be improved by studying the number of Neumann lines to which the
extremal points are connected. This number equals the number of Neumann
domains which share the same extremal point on their boundary and
we call it the valency of the extremal point (see also section \ref{sec:number_of_Neumann_domains}).
In addition, there are eigenfunctions for which all nodal domains
have a unique extremum. From the proof above we conclude that for
those eigenfunctions all Neumann domains obey \eqref{eq:lower_bound_on_outer_radius}.
Such eigenfunctions are given for example in \eqref{eq:explicit_eigenfunction}
and figure \ref{fig:Lense-like_and_star_like}.

There is no general upper bound on the outer radius of Neumann domains.
This can be demonstrated on the following family of separable eigenfunctions
on the unit torus,\\
 $\left\{ \cos\left(2\pi x\right)\cos\left(2\pi ny\right)\right\} _{n=1}^{\infty}$
(see figure \ref{fig:Lense-like_and_star_like}). All of those eigenfunctions
posses Neumann domains whose diameter equals $\nicefrac{1}{2}$.
\end{rem}

\subsection{On the restriction of eigenfunctions to Neumann domains\label{sec:lambda_2_conjecture}}

All Laplacian eigenfunctions have the following fundamental property
- their restriction to any nodal domain equals the first eigenfunction
of this domain. This has been used already in Pleijel's asymptotic
result for the nodal count \cite{Pleijel56}. It is therefore natural
to ask whether a similar statement holds for Neumann domains. The
restriction of an eigenfunction to one of its Neumann domains corresponds
to an eigenfunction with Neumann boundary conditions on that domain.
See also the discussion preceding proposition \ref{prop:unbounded_spectral_position}.
However, we provide here a counter-example, showing that the position
of the `global' eigenvalue in the spectrum of a single Neumann domain
is not always (i.e. for all manifolds) bounded from above. We remark
that the counter-example does not rule out that there are specific
classes of manifolds or domains $M$ for which there is an upper bound.\\

\begin{proof}
~[of proposition \ref{prop:unbounded_spectral_position}] Let $\mathbb{T}=[0,1]\times[0,1]$
be the unit flat torus with the corresponding Euclidean metric. Assume
by contradiction that there exists an $N\in\mathbb{N}$ such that
for all eigenfunctions $f$ and all Neumann domains $\Omega$ of $f$
\begin{equation}
\mathrm{pos}(\left.f\right|_{\Omega},\Omega)\leq N.
\end{equation}
By \cite{Kro_jfa92} we have the following bound
\begin{equation}
\forall\la,f,\Omega\,\,\,\,\,\lambda\leq\frac{8\pi N}{A\left(\Omega\right)},
\end{equation}
where $A\left(\Omega\right)$ denotes the area of the Neumann domain.
Summing $\lambda A\left(\Omega\right)\leq8\pi N$ over all Neumann
domains one gets $\lambda A\left(\mathbb{T}\right)\leq8\pi N\mu$
and we obtain that the number of Neumann domains $\mu$ obeys the
following lower bound
\begin{equation}
\mu\geq A\left(\mathbb{T}\right)\frac{\lambda}{8\pi N}.\label{eq:lower bound for Neumann count}
\end{equation}
\begin{figure}[htbp]
\centering{}\includegraphics[scale=0.18]{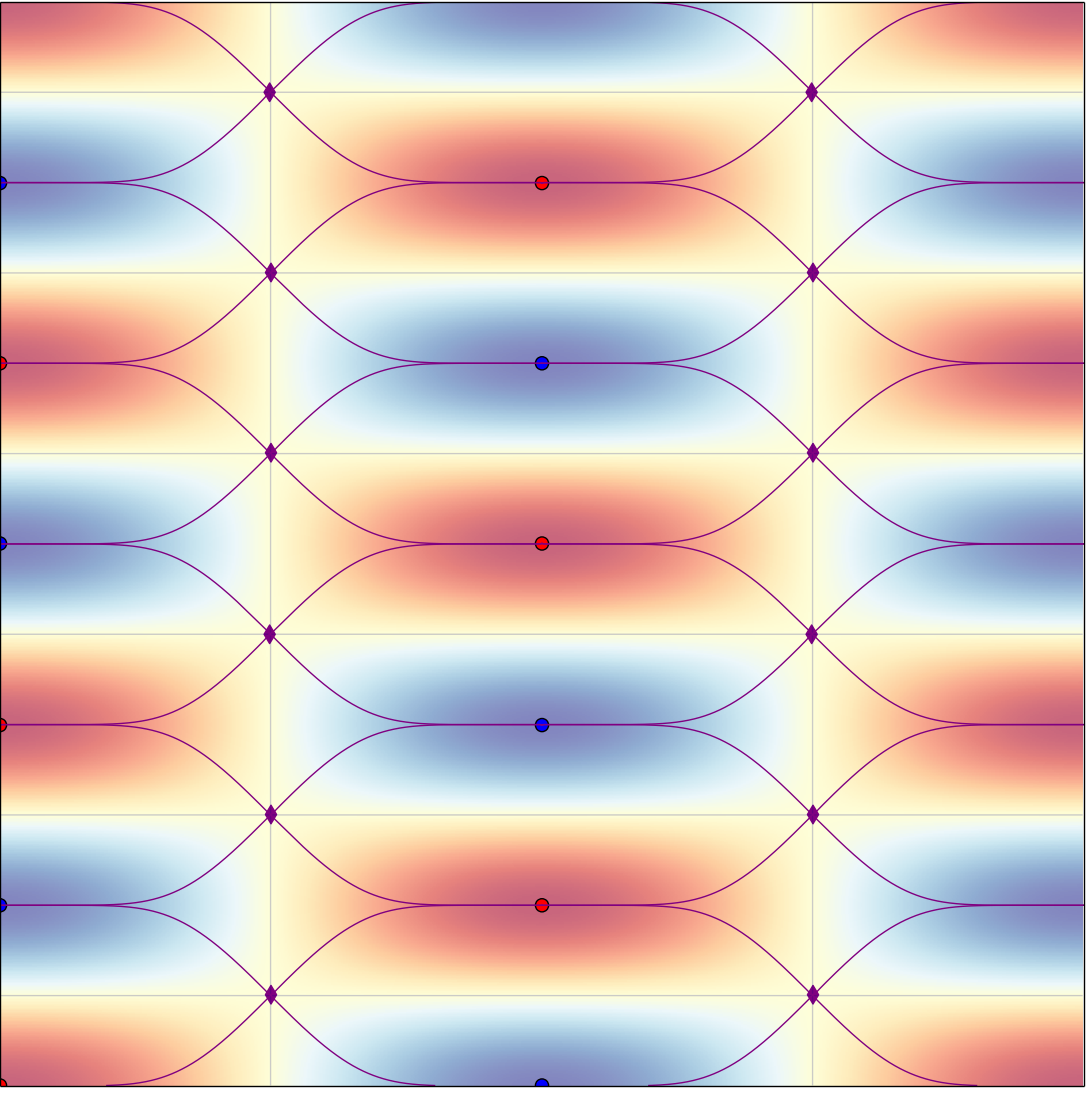}
\protect\caption{The Neumann lines of the unit torus eigenfunction $f\left(x,y\right)=\cos\left(2\pi x\right)\cos\left(6\pi y\right)$.
It is a member of infinite sequence of eigenfunctions which shows
that there is no uniform bound for $\mathrm{p}(\protect\la,f)$. \label{fig:Lense-like_and_star_like}}
\end{figure}

We now point on an example which contradicts the bound \eqref{eq:lower bound for Neumann count}.
Consider the following eigenfunction
\begin{equation}
f\left(x,y\right)=\cos\left(2\pi n_{x}x\right)\cos\left(2\pi n_{y}y\right),\label{eq:explicit_eigenfunction}
\end{equation}
with eigenvalue and number of Neumann domains
\begin{equation}
\begin{aligned}\lambda & =4\pi^{2}\left(n_{x}^{2}+n_{y}^{2}\right)\!,\\
\mu & =8n_{x}n_{y},
\end{aligned}
\end{equation}
respectively (see figure \ref{fig:Lense-like_and_star_like}). The
contradiction with \eqref{eq:lower bound for Neumann count} can be
easily seen if one chooses $n_{x}=1,\,n_{y}\gg1$.
\end{proof}
\noindent One may get an insight on the counter-example in the proof
above by investigating the shape of the Neumann domains obtained from
the choice $n_{x}=1,\,n_{y}\gg1$. The eigenfunction \eqref{eq:explicit_eigenfunction}
on the flat torus has Neumann domains of two distinguished shapes,
which we call lense-like, and star-like (figure \ref{fig:Lense-like_and_star_like})
. We show in the following that for a sufficiently large value of
$\nicefrac{n_{y}}{n_{x}}$ the eigenfunction restriction to a lense-like
Neumann domain does not equal the second eigenvalue of this domain.
\begin{lem}
\noindent \label{lem:counter_example_on_torus}Let $\mathbb{T}$ be
the unit flat 2-dim.\,torus and $f\left(x,y\right)=\cos\left(2\pi n_{x}x\right)\cos\left(2\pi n_{y}y\right)$
its eigenfunction with $n_{x},n_{y}\in\mathbb{Z}$. Let $\Omega$
be a lense-like Neumann domain of $f_{(n_{x},n_{y})}$. Then $\exists c>0$
such that $\nicefrac{n_{y}}{n_{x}}>c\,\,\,\Rightarrow\,\,\,\mbox{pos}\left(\left.f\right|_{\Omega},\,\Omega\right)>1$ \end{lem}
\begin{proof}
The major and minor axes of the lense-like Neumann domain, $\Omega$,
are of lengths $\ell_{x}=\nicefrac{1}{2n_{x}},\,\ell_{y}=\nicefrac{1}{2n_{y}}$.
This Neumann domain is convex and we may apply theorem 1.2(a) of \cite{JE00}.
According to that theorem, for a fixed value of $\ell_{x}$ (obeying
$\ell_{x}>\ell_{y}$) there is a constant $C$ ($\ell_{x}$ dependent)
such that the nodal set of the second eigenfunction is contained within
a vertical strip of width $2C\ell_{y}$ around the center of $\Omega$.
Namely, if $\varphi$ is the second eigenfunction then $\varphi\left(x,y\right)=0\,\,\Rightarrow\,\,\left|x\right|<C\ell_{y}$,
with the origin taken at the center of $\Omega$. Since in our case,
the nodal set of $\left.f\right|_{\Omega}$ is horizontal along $\Omega$
(see figure \eqref{fig:Lense-like_and_star_like}), we conclude that
for small enough value of $\ell_{y}$ the nodal set will not belong
to the allowed strip and therefore $\left.f\right|_{\Omega}$ cannot
be the second eigenfunction of $\Omega$.
\end{proof}

\section{the number of Neumann domains\label{sec:number_of_Neumann_domains}}

In this section we denote the number of Neumann domains by $\mu$
and the number of nodal domains by $\nu$. If the manifold has boundary,
we denote by $\mu^{\textrm{in}},\mu^{\textrm{bd}}$ the number of
inner and boundary Neumann domains, respectively and have $\mu=\mu^{\textrm{in}}+\mu^{\textrm{bd}}$.
For manifolds without boundary we already stated in corollary \ref{cor:Neumann_count_bdd_by_nodal_count},
which we prove below, that the number of Neumann domains is bounded
from below by half the number of nodal domains.\\

\begin{proof}
~[of corollary \ref{cor:Neumann_count_bdd_by_nodal_count}] Applying
theorem \ref{thm:topological_properties_no_boundary}\eqref{enu:thm-no-bdry-level-sets-3}
with $c=0$, we have that a Neumann domain contains only a single
non self-intersecting nodal line. Hence, each Neumann domain intersects
with exactly two nodal domains. It is possible that different Neumann
domains intersect with the same nodal domain. The number of nodal
domains is therefore bounded from above by $2\mu$.
\end{proof}
The same bound holds also for manifolds with boundary. Furthermore,
it may be slightly improved if separating between the count of boundary
and inner Neumann domains.\\

\begin{cor}
{[}of theorem \ref{thm-2}{]} \label{cor:Neumann_count_bdd_by_nodal_count-for_boundary}

Let $(M,g)$ be as in theorem \ref{thm-2} and $f$ a Morse eigenfunction
on $M$. Then $2\mu^{\textrm{in}}+\mu^{\textrm{bd}}\geq\nu$ and $2\mu\geq\nu$.\\
\end{cor}
\begin{proof}
~[of corollary \ref{cor:Neumann_count_bdd_by_nodal_count-for_boundary}]
From theorem \ref{thm-2}\eqref{enu:thm2-referring_to_thm1_2} we
deduce as in the proof of corollary \ref{cor:Neumann_count_bdd_by_nodal_count}
that inner Neumann domains intersect with exactly two nodal domains.
From theorem \ref{thm-2}\eqref{enu:thm2-sec-3} we deduce that boundary
Neumann domains intersect with a single nodal domain. Different Neumann
domains may intersect with the same nodal domain and hence $2\mu^{\textrm{in}}+\mu^{\textrm{bd}}\geq\nu$.
We also get $2\mu=2\mu^{\textrm{in}}+2\mu^{\textrm{bd}}\geq\nu$.
\end{proof}
\noindent The number of Neumann domains may also be studied by examining
the graph structure of the Neumann line set. The vertices of such
a graph are the critical points and the edges are the Neumann lines
connecting them. It is then natural to define the \emph{valency of
a critical point}, $\textrm{val}\left(x\right)$, as the number of
Neumann lines which are connected to $x$. The following discussion
is restricted for the case of manifolds without boundary. Combining
Euler's formula and Morse inequalities we get
\begin{equation}
\chi\left(M\right)=V-E+F=\left|\Min\right|-\left|\Sd\right|+\left|\Max\right|,\label{eq:Euler_formula}
\end{equation}
where $V,\,E,\,F$ are correspondingly the numbers of vertices, edges
and faces of the graph. The number of vertices is
\begin{equation}
V=\left|\Min\right|+\left|\Sd\right|+\left|\Max\right|\label{eq:number_of_vertices}
\end{equation}
and the number of edges obeys
\begin{equation}
E\leq4\left|\Sd\right|,\label{eq:upper_bound_on_edge_number}
\end{equation}
as at least one endpoint of each edge is a saddle point and all saddles
are of valency four. The faces correspond to Neumann domains, $F=\mu$,
and we therefore get that their number obeys
\begin{equation}
\mu\leq2\left|\Sd\right|.\label{eq:upper_bound_for_Neumann_count}
\end{equation}
If we further assume a Morse-Smale function we get equalities in both
\eqref{eq:upper_bound_on_edge_number} and \eqref{eq:upper_bound_for_Neumann_count}.
We now wish to obtain a lower bound on the number of Neumann domains.
Observe that
\begin{align}
E & =\frac{1}{2}\sum_{x\in\Cr}\textrm{val}\left(x\right)=\frac{1}{2}\left(4\left|\Sd\right|+\sum_{p\in\Xt}\textrm{val}\left(p\right)\right),\label{eq:lower_bound_on_edge_number}
\end{align}
where we used that saddles are of valency four. Plugging \eqref{eq:number_of_vertices},
\eqref{eq:lower_bound_on_edge_number} and $F=\mu$ in \eqref{eq:Euler_formula}
we get
\begin{align}
\mu\geq\frac{1}{2}\left|\Xt\right| & =\frac{1}{2}\sum_{p\in\Xt}\textrm{val}\left(p\right)=\frac{1}{2}\chi\left(M\right)+\frac{1}{2}\left|\Sd\right|.\label{eq:lower_bd_on_Neumann_count_in_terms_of_saddles}
\end{align}

We now assume that the manifold has boundary and is equipped with
Dirichlet boundary conditions (as in section \ref{sec:with_boundary}).
We note that each inner Neumann domain has a single minimum and a
single maximum on its boundary (theorem \ref{thm:topological_properties_no_boundary}\eqref{enu:thm-no-bdry-extremal-points})
and each boundary Neumann domain has either a minimum or a maximum
on its boundary (theorem \ref{thm-2}\eqref{enu:thm2-sec-1}). The
valency of a critical point equals the number of Neumann domains whose
boundary contain this critical point. We therefore obtain for manifolds
with boundary
\begin{equation}
2\mu^{\textrm{in}}+\mu^{\textrm{bd}}=\sum_{p\in\Xt}\deg\left(p\right),\label{eq:}
\end{equation}
which leads to
\begin{equation}
\frac{1}{2}\left|\Xt\right|\leq\frac{1}{2}\sum_{p\in\Xt}\deg\left(p\right)\leq\mu\leq\sum_{p\in\Xt}\deg\left(p\right)\leq4\left|\Sd\right|,\label{eq:-1}
\end{equation}
where the right inequality holds as each Neumann line emanating from
an extremum ends at a saddle point and each saddle point is connected
by Neumann lines to at most four different extremal points. The relations
above motivate the study of the extremal points valencies even if
just in the distributional sense.

\noindent Finally, let us discuss the asymptotics of the Neumann domain
count. The existence of subsequences of eigenfunctions whose nodal
count goes to infinity was recently proved \cite{GhoRezSar_gafa13,JunZel_arxiv13,JunZel_arxiv14}.
In \cite{GhoRezSar_gafa13} it was done for the arithmetic case and
in \cite{JunZel_arxiv13,JunZel_arxiv14} it was shown for a class
of non positively curved manifolds. From corollary \ref{cor:Neumann_count_bdd_by_nodal_count},
we conclude that in these cases there exists a subsequence of eigenfunctions
whose Neumann domain count goes to infinity as well. Furthermore,
numerical experiments suggest that the number of Neumann domains goes
to infinity as $\la\rightarrow\infty$. This is the case even for
sequences of eigenfunctions for which the number of nodal domains
is bounded (see for example figure \ref{fig:Stern-Example} which
describes the well-known example by Stern given in \cite{CourantHilbert_volume1}).
However, the statement above does not hold for all metrics. There
are known examples of metrics on the torus constructed by Jakobson
and Nadirashvili \cite{JakNad_jdg99}, which have subsequences of
eigenfunctions corresponding to eigenvalues $\la\rightarrow\infty$
with uniformly bounded number of critical points. As the saddle points
in this example are non-degenerate ones \cite{Jakobson_private},
the boundedness of number of saddles implies by \eqref{eq:upper_bound_for_Neumann_count}
that the number of Neumann domains for these subsequences is also
uniformly bounded. In other words, eigenfunctions corresponding to
arbitrarily high eigenvalues might have a small number of Neumann
domains.

\begin{figure}[htbp]
\centering{}\includegraphics[scale=0.18]{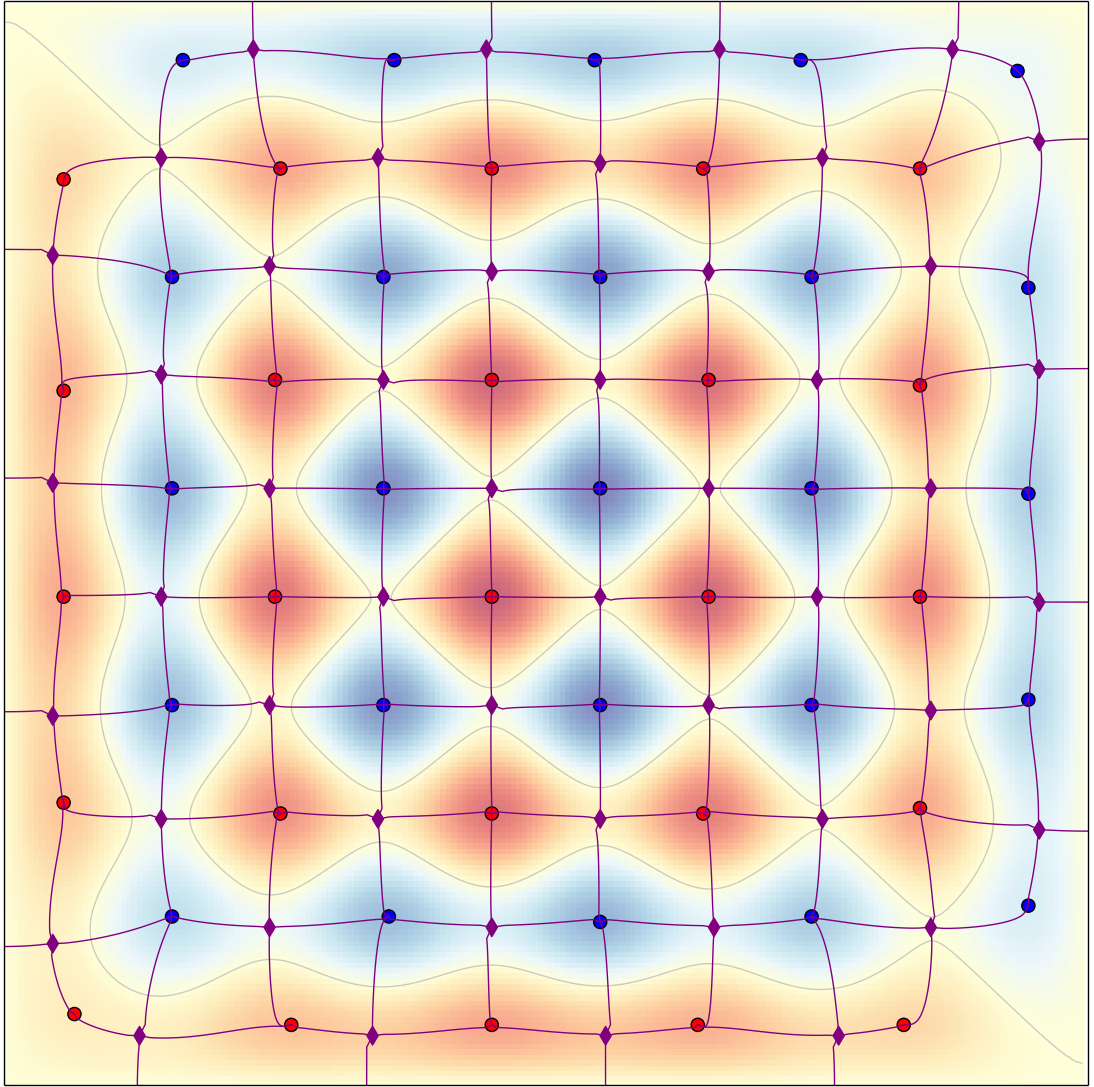}
\protect\caption{The Neumann lines for the eigenfunction $\sin(2rx)\sin\left(y\right)+\mu\sin(2rx)\sin\left(y\right)$
with $r=5,\,\mu\approx1$ on a square of edge size $\pi$ and Dirichlet
boundary condition. It belongs to a family of eigenfunctions with
only two nodal domains, but with number of Neumann domains which is
proportional to $r^{2}$. Cf.\,the example in page 396 of \cite{CourantHilbert_volume1}.\label{fig:Stern-Example}}
\end{figure}

\section{Summary}

\noindent This paper studies Laplacian eigenfunctions on surfaces
by investigating their Neumann domains. Given an eigenfunction, we
define Neumann lines and Neumann domains and show that they form a
partition of the manifold. Furthermore, we claim that this partition
is as natural as the partition dictated by the nodal set. However,
numerous essential questions that are being investigated for nodal
domains are open for Neumann domains. The current paper develops this
study by discussing and answering some of those questions.

\noindent Let us specify some points of comparison between Neumann
domains and nodal domains. From a topological point of view, Neumann
domains are simply connected (theorems \ref{thm:topological_properties_no_boundary}\eqref{enu:thm-no-bdry-simply-connected}
and \ref{thm-2}\eqref{enu:thm2-simply-connected}), whereas nodal
domains are not in general \cite{NazSod_ajm09,SarWig_arXiv13}. The
simplicity of the Neumann partition is also apparent in the eigenfunction
restriction to a Neumann domain, $\left.f\right|_{\Omega}$. Theorems
\ref{thm:topological_properties_no_boundary} and \ref{thm-2} show
that the structure of $\left.f\right|_{\Omega}$ cannot be too complex
in terms of the position and number of critical points and the nodal
set within $\left.f\right|_{\Omega}$. As $\left.f\right|_{\Omega}$
is also an eigenfunction of the domain $\Omega$ with Neumann boundary
conditions, its structural simplicity suggests that the position of
$\left.f\right|_{\Omega}$ in the spectrum of $\Omega$ cannot be
too high. A similar question for nodal domains is easy to answer and
for each nodal domain $D$ of $f$, it is known that $\left.f\right|_{D}$
is the first Dirichlet eigenfunction of $D$. This observation was
used by Pleijel \cite{Pleijel56} to obtain an asymptotic bound on
the nodal domain count. Similarly, answering the analogous question
for Neumann domains would help in estimating the number of Neumann
domains, as is discussed in section \ref{sec:lambda_2_conjecture}.
In particular, we already show that the number of Neumann domains
is bounded from below by half the number of nodal domains (corollary
\ref{cor:Neumann_count_bdd_by_nodal_count}) for the types of manifolds
we consider. \\
 It is well known that the number of nodal domains is affected by
the stability of the nodal set. This is apparent for example in the
case of multiple eigenvalues. Such eigenvalues may posses eigenfunctions
where nodal lines intersect. Perturbations of these eigenfunctions
may prevent these crossings and the intersecting lines resolve into
two separate nodal lines. The nature of this resolution of the intersection
crucially affects the topology and number of nodal domains and makes
their counting a difficult task. Neumann domains, however, show a
different behavior. A crossing of nodal lines always occurs at a saddle
point of the function and therefore it also coincides with a Neumann
line intersection. Such a Neumann line crossing is stable with respect
to perturbations and thus there is no change in the number of Neumann
domains when the eigenfunction is perturbed. This was already observed
in \cite{McDFul_ptrs13} and it was suggested that the Neumann line
pattern is relatively robust and hence the relative ease (in comparison
with nodal domains) of the Neumann domain count. Yet, there is an
additional phenomenon which complicates the count of Neumann domains.
Considering a multiple eigenvalue and some non Morse-Smale eigenfunction
which belongs to it, a perturbation might cause an appearance of a
new Neumann domain. Such a domain appears at the Neumann line which
connects some two saddle points and its volume may be arbitrarily
small. The purpose of theorem \ref{thm-lower_bound_on_outer_radius}
is to place a restriction on the number of such shrinking domains,
by providing a lower bound on the outer radius of some of the Neumann
domains.

\noindent Finally, we wish to point out open problems and possible
exploration directions for the study of Neumann domains. In the following
$M$ denotes a two dimensional compact manifold with or without boundary,
$\left(\lambda,f\right)$ denotes an eigenpair of the Laplacian on
$M$ and $\Omega$ is some Neumann domain of $f$.
\begin{enumerate}
\item Let $M$ be a 2-dimensional surface. For a Morse-eigenfunction, $f$,
of $M$, denote (see also proposition \ref{prop:unbounded_spectral_position}
and the preceding discussion)
\[
\textrm{p}\left(f\right):=\max_{\Omega}\left\{ \left.\mathrm{pos}\left(\left.f\right|_{\Omega},\Omega\right)~\right|\,\Omega\mbox{ is a Neumann domain of \ensuremath{f}}\right\} .
\]
What conditions on $M$ does one need to assure that $\left\{ \left.\textrm{p}\left(f\right)\right|~f~\textrm{is an eigenfunction}\right\} $
is either bounded or possesses a bounded subsequence? Such boundedness
imposes lower bounds and asymptotic results for the Neumann domain
count (see proof of proposition \ref{prop:unbounded_spectral_position}).
\item What are the asymptotics of the Neumann domain count? More specifically,
does the limit of $\left\{ \nicefrac{\mu_{n}}{n}\right\} _{n=1}^{\infty}$
exist, in general or for some classes of manifolds? If so, could it
be bounded from below?\\
 An easier task would be to bound $\liminf_{n\rightarrow\infty}\frac{\mu_{n}}{n}$
from below. Is it possible to obtain a Courant-like bound? Namely,
obtain an upper bound of the form $\mu_{n}\leq h\left(n\right)$,
with $h$ being some function (possibly linear).
\item Improving the inequality established in corollary \ref{cor:Neumann_count_bdd_by_nodal_count}
between the nodal count and the Neumann domain count or the lower
bound \eqref{eq:lower_bd_on_Neumann_count_in_terms_of_saddles} on
the Neumann domain count. This can be done, for example, by bounding
from below the valencies of extremal points (see discussion in section
\ref{sec:number_of_Neumann_domains}).
\item Bounding the total length of the Neumann line set in terms of the
eigenvalue.
\item Providing a global upper bound for the volume of a single Neumann
domain in terms of the eigenvalue.
\item Is it possible to improve the lower bound on the outer radius of a
Neumann domain in theorem \ref{thm-lower_bound_on_outer_radius}?
The main improvement might be to make this bound global, so that it
applies to all Neumann domains of the eigenfunction.
\item Provide an upper bound on the inner radius of a Neumann domain.
\end{enumerate}

\section{Acknowledgments}

\noindent We are grateful to Thomas Hoffmann-Ostenhof for several
enlightening discussions during the course of this work. We thank
those who listened to our ideas and helped us to sharpen them: Robert
Adler, Gregory Berkolaiko, Michael Berry, Mark Dennis, Stephen Fulling,
Dmitry Jakobson, Peter Kuchment, Dan Mangoubi, Uzy Smilansky, Mikhail
Sodin and Steve Zelditch. A special thanks goes to Tobias Hartnick
for the idea of proof in theorem \ref{thm:topological_properties_no_boundary}\eqref{enu:thm-no-bdry-simply-connected}
and to Michael Levitin for interesting discussions on the first problem
in our list of open questions. Alexander Taylor is warmly acknowledged
for providing the code with which the figures were generated. Finally,
we thank the referee for the careful reading of the manuscript and
their useful suggestions. R.B. was supported by ISF (grant No. 494/14),
Marie Curie Actions (grant No. PCIG13-GA-2013-618468) and the Taub
Foundations (Taub Fellow). D.F.~thanks the mathematics department
of the Technion for their hospitality.

\bibliographystyle{plain}
\bibliography{Morse_Theory_etal,Nodal_Domains_etal,Spectral-misc}
 \vspace{0.5cm}
 \textsc{Ram Band} (corresponding author)\\
 \textsc{Department of Mathematics, Technion - Israel Institute of
Technology, Haifa 32000, Israel}\\
 \texttt{ramband@technion.ac.il}\\
 \vspace{0.05cm}
\\
 \textsc{David Fajman}\\
 \textsc{University of Vienna}\\
 \texttt{David.Fajman@univie.ac.at}\\

\end{document}